\documentclass[11pt]{article}
\usepackage[utf8]{inputenc}

\makeatletter
\newcommand{\mylabel}[2]{#2\def\@currentlabel{#2}\label{#1}}
\makeatother
\usepackage{braket}
\usepackage{amsfonts,amssymb}
\usepackage{color}
\usepackage{amsthm,latexsym,amsmath,dsfont}

\usepackage[margin=1.2 in]{geometry}
\usepackage[ngerman,english]{babel}
\usepackage{booktabs}
\usepackage{siunitx}
\usepackage{natbib}
\usepackage{bbm}
\usepackage[normalem]{ulem}

\usepackage{extarrows}
\usepackage{accents}
\usepackage{footmisc}
\usepackage{graphicx}
\graphicspath{/Users/Ruhui/Dropbox/Chen_Jin_manuscript/figures}

\usepackage{hyperref}
\makeatletter
\def\eqref{\@ifstar\@eqref\@@eqref}
\def\@eqref#1{\textup{\tagform@{\ref*{#1}}}}
\def\@@eqref#1{\textup{\tagform@{\ref{#1}}}}
\makeatother

\usepackage[capitalise,nameinlink]{cleveref}
\setcitestyle{super,comma,square,numbers}
\usepackage{bbold}
\usepackage{mathtools}
\usepackage[dvipsnames]{xcolor}

\newtheorem{theorem}{Theorem}[section]

\newtheorem{lemma}[theorem]{Lemma}

\newtheorem{proposition}[theorem]{Proposition}
\newtheorem{corollary}[theorem]{Corollary}

\newtheorem{claim}[theorem]{Claim}
\crefname{figure}{Figure}{figures}

\theoremstyle{remark}
\newtheorem{remark}{Remark}[section]

\theoremstyle{definition}
\newtheorem{definition}{Definition}[section]

\usepackage[capitalise,nameinlink]{cleveref}
\setcitestyle{super,comma,square,numbers}
\usepackage{bbold}
\usepackage{mathtools}

\crefname{proposition}{Proposition}{propositions}
\crefname{lemma}{Lemma}{Lemmas}
\crefname{proposition}{Proposition}{propositions}
\crefname{definition}{Definition}{definitions}

\usepackage[capitalise,nameinlink]{cleveref}
\setcitestyle{super,comma,square,numbers}
\usepackage{bbold}

\usepackage[toc,page]{appendix}

\usepackage[T1]{fontenc}
\usepackage{tikz}
\usetikzlibrary{arrows,calc,positioning}
\tikzstyle{int}=[draw,minimum size=2em,text centered,text width=3cm]
\tikzstyle{sum}=[draw,shape=circle,inner sep=2pt,text centered,node distance=3.5cm]
\tikzstyle{summ}=[drawshape=circle,inner sep=4pt,text centered,node distance=3.cm]

\usepackage{braket,amsfonts,amssymb}
\usepackage[caption=false]{subfig}
\usepackage{array}
\usepackage[shortlabels]{enumitem}
\usepackage{pgfplots}
\usepackage{graphicx,epstopdf}
\usepackage{graphicx,epstopdf}
\usepackage{booktabs}
\usepackage{amsopn}
\usepackage{mathrsfs}  
\usepackage[mathscr]{euscript}
\usepackage{bm,nicefrac}
\usepackage{extarrows}

\usepackage{thmtools}

\usepackage{amsopn}
\usepackage{mathrsfs}  
\usepackage[mathscr]{euscript}
\usepackage{bm,nicefrac}
\usepackage{extarrows}

\def\S{\mathcal{S}}
\def\mA{\mathbf{A}}
\def\mB{\mathbf{B}}
\def\mC{\mathbf{C}}
\def\va{\mathbf{a}}
\def\vb{\mathbf{b}}

\def\vr{\mathbf{r}}
\def\mF{\mathbf{F}}
\def\mG{\mathbf{G}}

\def\mU{\mathbf{U}}
\def\mV{\mathbf{V}}
\def\mW{\mathbf{W}}
\def\mZ{\mathbf{Z}}

\def\mQ{\mathbf{Q}}
\def\Real{\mathbb{R}}

\def\mS{\mathbf{S}}
\def\pr{\mathbf{Pr}}
\def\nor{\mathcal{N}}
\def\vx{\mathbf{x}}

\def\vy{\mathbf{y}}
\def\vz{\mathbf{z}}
\def\sY{\bm{\mathcal{Y}}}

\def\sP{\bm{\mathcal{P}}}
\def\vf{\mathbf{f}}
\def\vg{\mathbf{g}}
\def\E{\mathbb{E}}
\def\exp{{\text{exp}}}
\def\ex{\mathbb{E}}
\def\eps{\varepsilon}

\def\mX{\mathbf{X}}
\def\mY{\mathbf{Y}}
\def\K{\bm{\mathcal{K}}}
\def\AK{\mA\K}
\def\AKb{\overline{\mA\K}}
\def\sT{\bm{\mathcal{T}}}
\def\W{\mathbb{W}}
\def\S{\bm{\mathcal{S}}}
\def\id{\mathbf{I}}
\def\mG{\mathbf{G}}
\def\mZ{\mathbf{Z}}

\def\bxi{{\bm{\xi}}}
\def\Wb{\mathbb{W}}
\def\C{\bm{\mathcal{C}}}
\def\veta{\bm{\eta}}

\def\vphi{\bm{\phi}}
\def\mPhi{\bm{\Phi}}
\def\vsig{\bm{\sigma}}
\def\mM{\mathbf{M}}
\def\sbY{\overline{\bm{\mathcal{Y}}}}

\def\mSig{\bm{\Sigma}}
\def\bern{\text{Ber}}

\def\O{\mathcal{O}}

\def\conv{\text{conv}}
\def\diam{\text{diam}}

\usepackage{bbold}
\usepackage{mathtools}
\usepackage{footmisc}

\begin{document}
\title{Tensor-structured sketching for constrained least squares%
}

\author{Ke Chen\thanks{Department of Mathematics, University of Texas, Austin, TX }
\quad and Ruhui Jin\footnotemark[1]}

\maketitle
\begin{abstract}
Constrained least squares problems arise in many applications. Their memory and computation costs are expensive in practice involving high-dimensional input data. We employ the so-called ``sketching'' strategy to project the least squares problem onto a space of a much lower ``sketching dimension'' via a random sketching matrix. The key idea of sketching is to reduce the dimension of the problem as much as possible while maintaining the approximation accuracy.

Tensor structure is often present in the data matrices of least squares, including linearized inverse problems and tensor decompositions. In this work, we utilize a general class of row-wise tensorized sub-Gaussian matrices as sketching matrices in constrained optimizations for the sketching design's compatibility with tensor structures. We provide theoretical guarantees on the sketching dimension in terms of error criterion and probability failure rate. In the context of unconstrained linear regressions, we obtain an optimal estimate for the sketching dimension. For optimization problems with general constraint sets, we show that the sketching dimension depends on a statistical complexity that characterizes the geometry of the underlying problems. Our theories are demonstrated in a few concrete examples, including unconstrained linear regression and sparse recovery problems.
\end{abstract}

\section{Introduction} 
Constrained optimization plays an important role in the intersection of machine learning \cite{BCN18}, computational mathematics \cite{NW06}, theoretical computer science \cite{B15}, and many other fields. We consider the least squares problem of the following form:
\begin{equation}
\label{QP}
\vx^*:= \arg\min_{\vx\in\C}\|\mA\vx-\vb\|_2^2,
\end{equation}
where $\C\subset \Real^p$ is the convex constraint set, and $\mA\in\Real^{n\times p}$ and $\vb\in \Real^n$ are respectively the coefficient matrix and vector. When the optimal solutions are not unique, we denote $\vx^*\in\Real^p$ as one of the minimizers.

A naive approach to solve this over-determined problem \eqref{QP} $(m \ll n)$ requires polynomial time in the ambient dimension $n$, which is not ideal as $n$ is remarkably high in large-scale optimization settings. \emph{Sketching} is a leading alternative to approximate a high-dimensional system with lower-dimensionional representations for less data memory, storage, and computation complexity. It introduces a random matrix $\mS$ of size $m \times n$, called the sketching matrix, to reduce the original program \eqref{QP} in a smaller problem:
\begin{equation}
\label{SQP}
\hat{\vx}:=\arg\min_{\vx\in\C}\|(\mS\mA)\vx-(\mS\vb)\|_2^2,
\end{equation}
where the sketched coefficients $\mS\mA, \mS\vb$ are merely of dimension $m$ rather than $n.$ 

In terms of error guarantees, we consider the following prediction error
\begin{equation}
\label{eqn:err}
\|\mA\hat{\vx}-\vb\|_2^2 \leq (1+ \eps)^2\|\mA\vx^*-\vb\|_2^2\,,
\end{equation}
with high probability for a pre-specified error criterion $\eps\in (0,1)$.
It is an interesting question how to design a sketching matrix $\mS$ with limited sketching dimension $m$ that can achieve small prediction error in the sense of \eqref{eqn:err}. This question has been well investigated in a line of work \cite{DMM06, RT08, MM13, CW13, ANW14, PW15, PW16} based on common dimensionality reduction methods like CountSketch \cite{CCC04}, sparse $0$-$1$ matrices \cite{DKS10}, Gaussian and sub-Gaussian matrices \cite{RV08, MPT07}, and FFT-based fast constructions \cite{AC06,T11}.

The coefficient data $\mA,\vb$ admit multi-linear (tensor) structures in many applications including spatio-temporal data analysis \cite{ BYL14, H15}, higher-order tensor decompositions \cite{DDV00, KB09}, approximating polynomial kernels \cite{KK12, PP13}, linearized PDE inverse problems \cite{CLNW19, JCCL19} and so on. Particularly, we focus on the data matrix $\mA\in\Real^{n_1n_2\times r}$ whose columns have tensor structure. 
In the sketching setting, we can utilize such structure in the original objective function to speed up forming the sketched problem \eqref{SQP} by designing sketching matrices with a corresponding tensor structure. In this paper, we study the sketching matrix $\mS\in\Real^{m \times n_1n_2}$ whose rows are tensor products of sub-Gaussian vectors. This design is natural to process the tensor data in the program due to the distributive property shown later in \eqref{distributive}. In particular, the cost of computing $\mS\mA$ drops significantly to $\O(m\, (n_1+n_2)\, p),$ from the cost $\O(m \,n_1n_2\, p)$ by a standard sub-Gaussian sketching. Previous works that share similar set-ups are \cite{SGTU18, CLNW19, RR20}, to which we will provide the details of comparisons in later discussions. Moreover, the input data $\mA, \vb$ have sparsity pattern in many practical situations. To reduce the computational complexity, we construct the sketching matrix with only a portion of nonzero entries. Specifically, we introduce a density level parameter $q \in (0,1)$ such that each random variable in the tensor components of $\mS$ is drawn to be zero with probability $1-q$. The computation cost further drops to $\O(m\, q^2\,(n_1+n_2)\, p).$

Apart from the row-wise tensorized sketches, there are other multi-linear random projections and sampling strategies that work well in practice.  
For sketching constructions employing fast matrix-vector multiplications, Pagh et al. \cite{P11, PP13} develop and analyze the TensorSketch method which uses fast Fourier transform (FFT) and CountSketch \cite{CCC04} techniques. This method is efficient while being applied to Kronecker products of vectors in the context of kernel machines. Later on, \cite{DSSW18} provides applications of TensorSketch to Kronecker product regressions as well as multimodal $p$-spline tensor sketching. 
Another line of works is represented by \cite{JKW19, MB19}. The authors respectively consider the so-called Kronecker FJLT. The analysis in both papers focuses on the subspace embedding property which can be seen as a stepping stone for sketching linear regressions. 
In the importance sampling regime, an efficient sublinear algorithm is provided by \cite{CPLP16} to sample the tensor CP alternating least squares (CP-ALS) problem \cite{KB09} by estimating the statistical leverage scores. Sparse sketching techniques are studied in \cite{LHW17} for low rank tensor CP and Tucker decompositions. 
For computational advantage from the tensor structure in the sketching matrices, please see the works \cite{CLNW19,PP13,JKW19} for details.
 
In the analysis of \cite{SGTU18, CLNW19, RR20, JKW19, MB19}, the authors rely on the Johnson-Lindenstrauss property \cite{JL84} to embed pairwise distances and derive concentrations. While such strategy is easy to apply and powerful for subspace embedding, it falls short of capturing the essential dimension of a subset and thus is suboptimal in most constrained optimization problems. Recently, Pilanci et al. \cite{PW15,PW16} obtain sharp guarantees for constrained convex programs via the Gaussian width of the constraint set. 
Another close prior work done by Bourgain et al. \cite{BDN15} focuses on the sparse JLT designed with a fixed number of non-zero entries per column. The authors develop fundamental analysis and discuss the relation between sparsity and embedding error for such class of sparse matrices, also using the Gaussian width parameter. 
This complexity parameter is capable of providing sharper bound via fine geometric argument. 

\subsection{Our contributions}
The current work aims to capture the essential geometry of the constraint set in \eqref{QP} and gain computational advantage from the tensor structure of the sketching matrix at the same time. 
Our main contributions have two components: for unconstrained linear regressions, we give a theoretical guarantee via an optimal Johnson-Lindenstrauss property for the row-wise tensorized sub-Gaussian sketches. For least squares with any choice of convex constraint sets, we adopt a variant measure of Gaussian width and provide an estimate on the sketching dimension for tensor-structured sketching matrices. This can be considered as a generalization of prior works \cite{MPT07,PW15,BDN15} where only unstructured sub-Gaussian sketching matrices are considered. To the best of our knowledge, this paper presents the state-of-the-art sketching dimensions for the row-wise tensor sketching matrices. 

For the sake of simplicity of the exposition, we first present the main result \cref{informal} for the sketching design constructed with row-wise tensorized Rademacher vectors \footnote{A random Rademacher vector has i.i.d. entries which take the values $-1, 1$ with equal probability $1/2.$} in the example of unconstrained linear regression problems. We would like to remark that more comprehensive results will be later shown in \cref{unconstrained} for a wide class of sub-Gaussian matrices and in \cref{sketch_dimension} for general constrained optimizations. For the proof of \cref{informal}, we refer the readers to \cref{unconstrained} in \cref{sec: 2}.
\begin{theorem}
\label{informal}
Let $n_1, n_2, p, m \in\mathbb{N}^+.$ Consider the linear regression problem \eqref{QP} $(n=n_1n_2)$ with coefficient matrix $\mA\in \Real^{n_1n_2\times p}$ and the constraint $\C=\Real^p$. Fix the error criterion $ \eps\in(0,1)$ and the failure probability $\delta \in (0,1/2).$ Let $\mS\in\Real^{m \times n_1n_2}$ be a matrix whose rows are independent tensor products of Rademacher vectors respectively of length $n_1$ and $n_2$. If the sketching dimension $m$ satisfies 
\begin{equation}
\label{opt_rank(A)}
m = \max\left(\O\left(\frac{\text{rank}^2(\mA)+\log^2 (1/\delta)}{\eps}\right), \O\left(\frac{\text{rank}(\mA)+\log (1/\delta)}{\eps^2}\right)\right),
\end{equation}
then the sketched solution $\hat{\vx}\in\Real^{p}$ in \eqref{SQP} satisfies \eqref{eqn:err} with probability exceeding $1-\delta$.
\end{theorem}
The estimate \eqref{opt_rank(A)} indicates that for an unconstrained least squares \eqref{QP} of large ambient dimension $n_1, n_2$, one can obtain an accurate approximated solution with high probability by solving the sketched problem \eqref{SQP} of dimension $m$, which can be as small as $\max(\O(\text{rank}^2(\mA)/\eps), \O(\text{rank}(\mA)/\eps^2))$.
According to Theorem $3.1$ in \cite{MSW19}, our estimation \eqref{opt_rank(A)} on the sketching dimension $m$ matches the optimal result of the row-wise tensor-structured Rademacher sketching matrix. Moreover, the second term in \eqref{opt_rank(A)} even hits the sharpest bound for unstructured sketches as in \cite{PW15}. We conclude that the row-wise tensor sketching strategy possesses both computational advantage and sufficient accuracy.

\subsection{Related work}
Sketching designs of row-wise tensor structure are previously studied in \cite{SGTU18, CLNW19, RR20} with applications for data memory reductions, tensor decompositions and linear inverse problems. Through the lens of unconstrained linear regressions, by applying a standard covering net technique, their results are suboptimal. The sketching sizes $m$ in the mentioned works are bounded below by $\text{rank}^8(\mA), \text{rank}^6(\mA), \text{rank}^4(\mA)$ respectively for qualified sketching to the linear regression in \eqref{QP}. Our work improves the estimate of $m$ to a linear or quadratic dependence on the rank of $\mA.$ 

A similar construction of the sketching matrix is considered in Meister et al. \cite{MSW19}, in which they form the sketching as the composition of a tensorized Rademacher matrix and a sparse $\{0,\pm1\}$-valued random matrix called CountSketch \cite{CCC04}. 
In comparison, we consider a wider class of sub-Gaussian random variables with adjustable density $q \in [0,1]$ that bridges the gap between sparse and dense sketching matrices. Hence the construction in \cite{MSW19} would fit into our sketching design framework.
Although the ultrasparse CountSketch in  Meister et al. \cite{MSW19} may enable faster running time, the theoretical analysis for the whole group of tensor-structured sketches, even for dense matrices, is barely comprehensive. In terms of accuracy, the authors of \cite{MSW19} show their design is optimal in vector-based embeddings up to logarithmic factors. In particular, by applying Theorem $2.1$ of \cite{MSW19} and a covering net argument, their work gives a bound of the sketching dimension 
\begin{equation}
\label{eqn:MSW19}
m \geq \max\left(\mathcal{O}\left(\frac{\text{rank}^2(\mA)+\log^2(1/\delta\, \eps)}{\eps}\right), \mathcal{O}\left(\frac{\text{rank}(\mA)+\log(1/\delta\, \eps)}{\eps^2}\right)\right),
\end{equation}
to restrict the sketching error $\eps$ in the sense of \eqref{eqn:err} with probability $1-\delta.$ Even though \eqref{eqn:MSW19} is shown to be optimal for their specific sparse sketching, their result is unknown for other dense matrices and for general constraint optimizations.
In comparison, we introduce a density level parameter $q$ to adjust the number of nonzeros in the sketching matrix. Additionally, our estimate $\max\left(\O\left(\frac{\text{rank}^2(\mA)+\log^2(1/\delta)}{\eps\, q^2}\right),\O\left(\frac{\text{rank}(\mA)+\log(1/\delta)}{\eps^2\, q^4}\right)\right)$ derived from \cref{unconstrained} is able to match their result in \eqref{eqn:MSW19} for any constant $q\in(0,1)$. 

Another related work is by Pilanci and Wainwright \cite{PW15}. In the focus of constrained least squares, the paper established a sharp sketching dimension bound employing the Gaussian width complexity for unstructured sub-Gaussian sketching matrices. In contrast, we study the case when the sketching matrices are constructed as row-wise tensor products of sub-Gaussian matrices given their computational advantages on tensor data coefficients. However, the tensorized sketching matrix introduces higher-order chaos and poses challenges to the sketching dimension estimation. Regarding the theoretical analysis, we develop a sketching dimension bound in \cref{sketch_dimension} via a modified complexity parameter $M$ defined in \cref{M def} \eqref{M-complexity}. This $M$-complexity is larger than its Gaussian width counterpart for the unstructured sketching case and may lead to suboptimal estimate. 

\subsection{Notations}
In the paper, we denote by $\|~ \|_{2}$ and $\|~\|_{\infty}$ respectively the $\ell_2$ and $\ell_\infty$ norms of a vector, and $\|~ \|, ~\|~\|_{F}$ respectively the spectral and Frobenius norm of a matrix. The symbol $\pr$ denotes the probability of an event, and the notion $\E$ is the expectation of a random variable. A $n$-dimensional vector following the distribution $\bern(q)^n$ has independent $\{0,1\}$-valued entries which take the value $1$ with probability $q.$
We use the calligraphic font $\bm{\mathcal{X}}$ to denote vector sets, the Roman script uppercase letter $\mX$ to denote matrices, the Roman script lowercase letter $\vx$ to denote vectors, and simple lowercase letter $x$ to denote a scalar entry. We write $\id_n \in \Real^{n \times n}$ as the $n$ by $n$ identity matrix. We denote $c, C$ as absolute constants whose values may change from time to time. A constant is universal if its value does not depend on any other parameters. 

\subsection{Organization of the paper}
The rest of the paper is organized as follows. 
\cref{sec: 2} starts with the reasoning for the sketching design and follows the two main results \cref{JL} and \cref{sketch_dimension} respectively about the embedding property and sketching dimension estimation results. We conclude the section with direct applications for two common classes of optimization problems.
In \cref{sec: 3}, we show an important intermediate result \cref{sup_error} about the supremum of embedding error of the tensorized sub-Gaussian processes and give the complete proof of \cref{sketch_dimension}. 
The proofs for \cref{JL} and \cref{sup_error} are then illustrated in \cref{sec: 4}. These proofs are built upon a crucial concentration result presented in Section 4.1.
\cref{sec: numerics} contains the numerical performance of the proposed sketching construction and provides empirical support for the theory.
We conclude and discuss future research directions in \cref{sec: conclusion}.

\section{Sketching dimension estimation}
\label{sec: 2}
The main goals of \cref{sec: 2} are analyzing the vector-based embedding property of the tensor-structured sub-Gaussian sketching matrices and developing a sketching size estimation for the constrained least squares problem \eqref{SQP}. 

We begin with some linear algebra and probability background. 
\begin{definition}
Given matrices $\mX\in\mathbb{R}^{M_1\times N_1}$ and $\mY\in\mathbb{R}^{M_2\times N_2}$, the Kronecker product of $\mX$ and $\mY$ is defined as
\begin{equation}
\mX \otimes \mY= \left[                 
\begin{array}{cccc}   
x_{1,1}\mY &x_{1,2}\mY  &\dots&x_{1,N_1}\mY \\  
\vdots&\vdots &\ddots&\vdots\\
x_{M_1,1}\mY &x_{M_1,2}\mY  &\dots&x_{M_1,N_1}\mY  \\  
\end{array}
\right]\in\mathbb{R}^{M_1M_2\times N_1N_2}.
\end{equation}
The Kronecker product satisfies the distributive property:
\begin{equation}
\label{distributive}
\mW\mX\otimes \mY\mZ = (\mW\otimes \mY)(\mX\otimes \mZ).
\end{equation}

Assuming $M_1 = M_2$ and $N_1 = N_2,$ the Hadamard product is defined as 
\begin{equation}
\mX\circ \mY \in \Real^{M_1 \times N_1},
\end{equation}
which is the element-wise product of the matrices. 
\end{definition}

\begin{definition}
\label{psi_2}
The sub-Gaussian norm of a random variable $x \in \Real$, denoted by $\|x\|_{\psi_2},$ is defined as:
$$
\|x\|_{\psi_2} = \inf\{t>0: ~~ \ex\, \exp(x^2/t^2) \leq 2\}.
$$
A random variable is sub-Gaussian if it has a bounded $\psi_2$ norm.
\end{definition}
We note that all normal random variables and all bounded random variables are sub-Gaussian.

\subsection{Design of the sketching matrix}
\label{design}
Tensor structure is ubiquitous in applied mathematics \cite{KB09}, statistics \cite{AGHKT14}, and deep learning \cite{CSS16}. Many datasets are naturally arranged with several attributes and can be represented as multi-arrays, namely tensors. Such structure also appears in the optimization problem \eqref{QP} in practical applications. We initiate our sketching matrix design with two motivating examples: CANDECOMP/PARAFAC (CP) tensor decomposition and linearized PDE inverse problems.

CP tensor decomposition is an important tool for large-scale data analysis. The workhorse algorithm in fitting CP tensor decomposition is the alternating least squares, which solves the following convex optimization problem:
\begin{equation}
\label{eqn:als}
\arg\min_{\mX\in \mathbb{R}^{R\times n }}\| \mA \mX - \mB \|_F \,,
\end{equation} 
where $\mA$ and $\mB$ are data matrices that are flattened from certain tensors. The flattening process forces columns of $\mA$ to have tensor structures. In particular, let $\va$ be a column of $\mA$, then it can be rewritten as a Kronecker product:
\[
\va = \bigotimes_{\ell=1}^d \va^{(\ell)}\,,
\]
where $d+1$ is the order of the original tensor from which $\mA$ is unfolded. Due to the high computation cost of solving \eqref{eqn:als}, sketching is used as a tool of solving CP tensor decomposition, see \cite{BBK18}. To relate the example to our work, we consider the $d=2$ case.

Another interesting example is the class of linearized PDE inverse problems. One famous case is the Electrical Impedance Tomography (EIT), which infers the body conductivity images from boundary measurements of voltage and current density on the surface. The linearization of EIT problem leads to the following Fredholm equation of the first kind:
\begin{equation}\label{eqn:ip}
\int f_{i_1}(y)\, g_{i_2}(y)\, \sigma(y) \,dy = \text{data}_{i_1,i_2}\,, \quad \forall 1\leq i_1\leq n_1,\,  1\leq i_2\leq n_2.
\end{equation}
Here $y$ is the spatial variable and $\sigma$ is the image to reconstruct. The functions $f_{i_1},g_{i_2}$ are known from physical understanding of EIT, and $\text{data}_{i_1,i_2}$ are observed data. The subscript $i_1$ and $i_2$ are indices for where the current is applied on the body surface and where the voltage is measured. In practice, one has to place electric nodes $y_j$ on various places to produce sufficient data for the reconstruction, leading to a large number of equations with the same structure as in \eqref{eqn:ip}. A numerical discretization of \eqref{eqn:ip} then produces a linear equation
$
\mA \vx  = \vb,
$
where $a_{i,j} = f_{i_1}(y_j)\, g_{i_2}(y_j)$, $x_j = \sigma(y_j), b_i = \text{data}_{i_1,i_2}$ and the index $i$ is associated with a pair of indices $(i_1, i_2).$ The above over-determined system is usually solved as an optimization problem \eqref{QP}. Each entry of $\mA$ is a product of functions $f_{i_1}$ and $g_{i_2},$ and entries in one column share the same dependence on the spatial variable $y_j.$ It can be shown that such data structure is equivalent to having Kronecker structure for all columns of $\mA$, that is, a column $\va$ in $\mA$ can be written as:
\[
\va = \vf \otimes \vg\in\Real^{n_1n_2} \,,
\]
for some column vectors $\vf \in \Real^{n_1}$ and $\vg\in\Real^{n_2}.$ For more details of tensor structures in linearized inverse problem, we refer the readers to \cite{CLNW19}.

In the aforementioned two examples, we see that tensor structure appears in the the columns of input matrix $\mA$ in the least squares objectives. When implementing the sketching strategy on the data matrix, the cost of the matrix multiplication $\mS\mA$ can be drastically reduced if $\mS$ has a consistent tensor structure in its rows. To further reduce the time complexity, we consider the sketching constructed with only a portion of nonzero entries. 
To this end, we design the sketching matrix $\mS$ in the following form:
\begin{equation}
\label{sketch}
\mS =
\frac{1}{\sqrt{m}}\, 
\left[
\begin{array}{c}
\veta_1^\top\otimes \bxi_1^\top\\
\vdots\\
\veta_k^\top\otimes \bxi_k^\top\\
\vdots\\
\veta_m^\top\otimes \bxi_m^\top\\
\end{array}
\right]\in \Real^{m \times n_1n_2}\,,
\end{equation}
where $\veta_k$ and $\bxi_k$ are independent copies of $\veta \in \Real^{n_1}, \bxi\in \Real^{n_2}.$ 
The random vectors $\veta$ and $\bxi$ are set to be
\begin{equation}
\label{sketch ii}
\veta = \frac{1}{\sqrt{q}}\, \left(\vphi^{(1)} \circ \vsig^{(1)}\right)\in \Real^{n_1}, \quad \bxi = \frac{1}{\sqrt{q}}\, \left(\vphi^{(2)} \circ \vsig^{(2)}\right)\in \Real^{n_2},
\end{equation}
where $\vphi^{(1)} \in \Real^{n_1}, \vphi^{(2)} \in \Real^{n_2}$ are vectors with two sets of i.i.d. zero-mean, unit variance sub-Gaussian variables. The vectors $\vsig^{(1)}$ and $\vsig^{(2)}$ follow the distributions $\bern(q)^{n_1}$ and $\bern(q)^{n_2}$ respectively with $q\in(0,1]$. We call $q$ as the density level. \footnote{We note that the density level $q$ is the percentage of nonzeros in expectation for each tensor factor of the sketching matrix $\mS.$ As a result, the nonzero percentage of $\mS$ is $q^2$ in expectation.}

\subsection{Main result}
\label{main result}
\subsubsection{Optimal JL property of tensorized sketching matrices}
The Johnson-Lindenstrauss property is considered the cornerstone for developing theoretical analysis of a sketching construction. The celebrated JL lemma \cite{JL84, LN17} shows that a finite set of high-dimensional points $\sY$ can be mapped to a space of (optimally) lower dimension $\mathcal{O}(\log\left(|\sY|\right)/\eps^2)$ within $(1 \pm \eps)$ distortion, where $|\sY|$ is the cardinality of $\sY$. Common choices of such maps are the class of sub-Gaussian matrices \cite{MPT07}. 

We present the JL property of the tensor-structured sub-Gaussian sketching matrices in \cref{JL}. A highlight of this main result is that the obtained embedding dimension \eqref{jl m} is optimal for our proposed tensor-structured sketches. This outcome is particularly useful in deriving guarantee for sketching unconstrained linear regressions, see \cref{unconstrained} for details.
\begin{theorem}
\label{JL}
Let $n_1, n_2, m\in \mathbb{N}^+.$ Fix a finite set $\sY \in\Real^{n_1n_2}$ of cardinality $|\sY|$ and parameters $\eps,\delta \in (0,1).$ Suppose the sketching matrix $\mS$ defined in \eqref{sketch}-\eqref{sketch ii} has the embedding dimension
\begin{equation}
\label{jl m}
m \geq C\, \max\left(\frac{\log^2(2\,|\sY|/\delta)}{\eps}, \frac{\log(2\,|\sY|/\delta)}{\eps^2}\right),
\end{equation}
then
 \[\pr\left(\|\mS\vy\|_2^2=(1\pm \eps)\,\|\vy\|_2^2, ~~\forall ~\vy\in  \sY \right) \geq 1-\delta.\] Here, $C>0$ \footnote{The constant $C$ depends on the largest $\psi_2$ norm of entries in $\vphi^{(1)}, \vphi^{(2)}$ and density level $q$, recalling the definition in \eqref{sketch}-\eqref{sketch ii}. We refer the readers to \eqref{C jl} in \cref{JL proof} for the explicit choice of $C$.} in \eqref{jl m} is a finite number.
\end{theorem}
The proof of \cref{JL} is shown in \cref{JL proof}. 

Regarding the embedding dimension bound $m$ \eqref{jl m}, we emphasize that the first term in maximal function exhibits merely linear dependence on the inverse of the error criterion $1/\eps$, while the second term in fact matches the best-known result for any oblivious sketches. An interesting lower bound in Theorem 3.1 of \cite{MSW19} proves that a row-wise tensor-structured sketching matrix would fail to have a qualified embedding if the embedding dimension $m$ is smaller than the quantity \eqref{jl m}. This implies that our embedding dimension bound in \cref{JL} is optimal. 

\subsubsection{Sketching for constrained least squares}
The above result \cref{JL} can be applied to develop sketching dimension for unconstrained linear regression problems but fails to tackle constrained least squares due to its lack of characterization of the underlying geometry of the constrained sets. We hence explore a new geometric approach and proceed to build a sketching dimension estimation for the tensor-structured $\mS \in \Real^{m \times n_1n_2}$ \eqref{sketch}-\eqref{sketch ii} in constrained convex optimization problems.

In convex analysis, convex cones help represent the optimality condition for the program. The tangent cone of the optimum $\vx^*$ with the convex constraint $\C$ is defined as:
\begin{equation}
\label{cone}
\K:= \{\mathbf{\Delta} = t(\vx-\vx^*)\in \Real^p, \quad\text{for}~t\geq 0 ~\text{and}~ \vx \in \C\}.
\end{equation}
We focus on the transformed cone $\mA\K\subset\Real^{n_1n_2}$ to measure the gap between the prediction error $\|\mA\hat{\vx}-\vb\|_2^2$ and the minimal error $\|\mA\vx^*-\vb\|_2^2$. 
Of particular interest, our analysis exploits the $M$-complexity (\cref{M def}) of the normalized transformed cone: $\AKb,$ where the overline symbol $\overline{\phantom{AA}}$ denotes the normalization of a set, i.e.
\begin{equation*}
\AKb = \left\{\frac{\vy}{\|\vy\|_2}\in \Real^n~~\Big\vert ~~\text{for}~\vy\in\AK\subset \Real^{n_1n_2}, \vy \neq {\bf 0}\right\}.
\end{equation*}
To formally introduce the $M$-complexity term, we start with some classical concepts in generic chaining \cite{T05}. Given a metric space $(T, d),$ an admissible sequence $\{\pmb{\mathcal{A}}_s\}_{s\in \mathbb{N}}$ of $T$ is a partition of $T$ such that $|\pmb{\mathcal{A}}_0| = 1$ and for $s \geq 1$, $|\pmb{\mathcal{A}}_s| \leq 2^{2^s}.$

\begin{definition}[$\gamma$-functionals \cite{T05}]
Suppose $\alpha \geq 1,$ we define the $\gamma_\alpha$-functionals by 
\begin{equation*}
\gamma_\alpha(T,d) := \inf\sup_{t \in T}\sum_{s\in \mathbb{N}} 2^{s/\alpha}\diam(\mathcal{A}_s(t)),
\end{equation*}
where $\{\pmb{\mathcal{A}}_s\}_{s\in \mathbb{N}}$ is any admissible sequence of $T$, $\mathcal{A}_s(t)$ is an element in partition $\pmb{\mathcal{A}}_s$ that contains $t$ and $\diam(\cdot)$ denotes the diameter of a set. The infimum is taken over all admissible sequences.
We use the short-hand notation $\gamma_\alpha(T)$ for $T$ being a set on the Euclidean space and $d$ being the $\ell_2$ distance. 
\end{definition}

The $M$-complexity parameter is built upon the computation of $\gamma$-functionals. 
\begin{definition}[$M$-complexity \cite{GK20}]
\label{M def}
For a set $\sY \subset \Real^{n},$ the $M$-complexity of the normalized set $\sbY \subset \S^{n-1}$ w.r.t. semi-norms $\|\cdot\|_{\text{g}}, \|\cdot\|_{\text{e}}$ is defined as:
\begin{equation}
\label{M-complexity}
M^{\text{(g,e)}}(\sbY):= \inf\left\{\gamma_1(\sT,  \|\cdot\|_{\text{e}})+\gamma_2(\sT,  \|\cdot\|_{\text{g}}) ~~ \vert  ~~\sT \subset \Real^n, \text{such~that~}\sbY \cup \{\mathbf{0}\}\subset \text{conv}(\sT)  \right\},
\end{equation}
where $\conv(\cdot)$ refers to the convex hull. For brevity, we write $M(\sbY)$ to denote $M^{(2,2)}(\sY)$ w.r.t the $\ell_2$ Euclidean distance.
\end{definition}

\begin{remark}
We adopt the notion of $M$-complexity from Definition 2.6 in \cite{GK20} by setting $t=1$. For the computation of $M(\sbY)$, one needs to measure the size of an optimal skeleton set $\sT$ that covers $\sbY$ and the origin, rather than the set $\sbY$ itself. We employ this calculation strategy in our estimation result \cref{sketch_dimension}, due to the convexity of least squares functions and the fact that all extreme value points on $\sbY$ are controlled within the covering skeleton. Although finding such optimal skeleton set remains an open problem, any skeleton set that covers $\sbY$ and the origin is sufficient to give an upper bound for $M(\sbY)$. For example, we later show an explicit skeleton construction in \cref{sparse recovery} for $\ell_1$-constrained sparse recovery problems. 
\end{remark}

In comparison with the $M$-complexity, we introduce another interesting complexity parameter: Gaussian width \cite{G85}. Gaussian width is a widely used notion in statistical learning theory and geometric analysis. It was shown to help establish an optimal bound for unstructured sub-Gaussian sketches \cite{PW15}. One important property is that Gaussian width is quantitatively equivalent to the $\gamma_2$-functionals up to some constant for the powerful majorizing measure theorem \cite{T96}. 

\begin{definition}[Gaussian width \cite{G85}]
The Gaussian width of the normalized set $\sbY \subset \S^{n-1}$ is defined as
\begin{equation}
\label{WY}
\Wb(\sbY) = \E \sup_{\vy \in \sbY} \langle \mathbf{n},\vy\rangle,
\end{equation}
where $\mathbf{n}\in \Real^{n}$ is a random vector drawn from the normal distribution $\nor(\mathbf{0}, \id_n).$
\end{definition}

\begin{remark}
\label{M-complexity remark}
The $M$-complexity of a set is strictly bigger than its Gaussian width counterpart. From \cref{M def}, $M$ contains not only the $\gamma_2$-functional which is quantitatively equivalent to the Gaussian width measure, but also a larger $\gamma_1$-functional component. However, in the current research stage, we have less understanding to explicitly evaluate the $\gamma_1$ term.  
\end{remark}

With the necessary preliminaries in place, we now illustrate the main result about the sketching dimension $m$ to guarantee a qualified output for a constrained convex program. 

\begin{theorem}
\label{sketch_dimension}
Let $n_1, n_2, p, m \in\mathbb{N}^+.$ Consider the constrained least squares \eqref{QP} ($n = n_1n_2$) with coefficient matrix $\mA\in \Real^{n_1n_2\times p},$ the vector $\vb\in\Real^{n_1n_2},$ the optimal solution $\vx^*\in\Real^p$ and the tangent cone $\K\subset \Real^p$ \eqref{cone}. Fix the parameters $\eps\in(0,1), \delta\in(0,1/2).$ If the sketching matrix $\mS\in\Real^{m \times n_1n_2}$ defined in \eqref{sketch}-\eqref{sketch ii} has the sketching size $m$ satisfying
\begin{equation}
\label{sketch_dimension_term}
m \geq \max\left(C\, \frac{\left(M(\AKb) \right)^2\, \log^2(1/\delta)}{\eps^2}, \quad 64^2\right),
\end{equation}
for some finite constant $C>0$ \footnote{The constant $C$ depends on the largest $\psi_2$ norm and density level of $\mS$. For the explicit choice of $C$, please see \eqref{C} in \cref{sketch_dimension proof}.}, then the sketched solution $\hat{\vx}\in\Real^p$ of \eqref{SQP} suffices to have
\begin{equation}
\label{eps_distort}
\|\mA\hat{\vx}-\vb\|^2_2\leq (1+\eps)^2\, \|\mA\vx^* -\vb\|_2^2
\end{equation}
with probability exceeding $1-\delta.$ 
\end{theorem}
The proof of \cref{sketch_dimension} is in \cref{sketch_dimension proof}. 

To interpret the above sketching dimension estimate for constrained least squares, we compare it with the unconstrained case result \cref{informal} which is derived separately via the JL property. It turns out the estimate in \cref{sketch_dimension} matches the worst-case scenario of the unconstrained case shown in \eqref{opt_rank(A)} \cref{informal}. Specifically, the sketching dimension has quadratic dependence on both $1/\eps$ and $\log(1/\delta)$. 
Moreover, the sketching dimension has quadratic dependence on the $M$-complexity, which replaces the Gaussian width parameter $\W(\AKb)$ in the well-known unstructured sub-Gaussian sketching result, see Theorem 1 in \cite{PW15}. Like Gaussian width, this new $M$-complexity suffices to capture the geometry of constrained optimizations. 
However, it gives a suboptimal estimate for the unconstrained case, where the resulting bound for $m$ is larger than $\O(\text{rank}(\mA))$. We suspect such sub-optimality is unavoidable due to the higher-order tensor structure in the sketching matrix. 

\subsection{Concrete case studies}
\label{applications}
In this subsection, we provide the case studies for different kinds of optimization problems utilizing the two main results \cref{JL} and \cref{sketch_dimension} shown above. 
In particular, in the example of unconstrained linear regressions, we establish a theoretical guarantee by applying \cref{JL} with standard covering net method to achieve a subspace embedding.

In terms of constrained convex programs with versatile geometric landscapes, \cref{sketch_dimension} offers a unified framework via calculation of the complexity $M(\AKb)$ based on the geometry of the constrained tangent cone $\K$ and the data matrix $\mA$.
For the sparse recovery with $\ell_1$-constraint problem, we give an explicit computation for $M(\AKb)$ and obtain a sketching size estimation as a direct consequence from \cref{sketch_dimension}.

\subsubsection{Unconstrained linear regression}
Unconstrained linear regression is the most common example of the convex program \eqref{QP}, by setting $\C$ to be the whole Euclidean space $\Real^p.$ The following corollary shows that it suffices to have $\max(\O(\text{rank}^2(\mA)/\eps), \O(\text{rank}(\mA)/\eps^2))$ rows in the sketching matrix to sketch an unconstrained linear regression problem. We claim such estimation for the row-wise tensor sub-Gaussian sketch \eqref{sketch} is sharp because it is derived from the optimal JL result following a covering net argument, which is a standard approach tackling unconstrained linear regressions. 

\begin{corollary}
\label{unconstrained}
Let $n_1, n_2, p, m \in\mathbb{N}^+.$ Consider the linear regression problem \eqref{QP} with data matrix $\mA\in \Real^{n_1n_2\times p}$ and the constraint $\C=\Real^p$. Fix the parameters $\eps\in(0,1)$ and $\delta \in (0,1/2).$ If the sketching matrix $\mS \in \Real^{m \times n_1n_2}$ defined in \eqref{sketch}-\eqref{sketch ii} has the sketching size
\begin{equation}
\label{unconstrained m}
m \geq C\,\max\left(\frac{\text{rank}^2(\mA)+\log^2(1/\delta)}{\eps}, \frac{\text{rank}(\mA)+\log(1/\delta)}{\eps^2}\right),
\end{equation}
then with probability exceeding $1-\delta,$ the sketched solution $\hat{\vx}\in\Real^{p}$ in \eqref{SQP} satisfies the accuracy as in \eqref{eps_distort}. Above, $C>0$ in \eqref{unconstrained m} is a finite constant.
\end{corollary}

\begin{proof}[Proof of \cref{unconstrained}]
The proof is based on \cref{subspace embed} and resembles the proof of Theorem 2.3 in \cite{W14}. 
\begin{lemma}
\label{subspace embed}
Follow the same set-up in \cref{unconstrained}. Draw the sketching matrix $\mS \in \Real^{m \times n_1n_2}$ with $m$ satisfying
\begin{equation*}
m \geq C\,\max\left(\frac{\text{rank}^2(\mA)+\log^2(1/\delta)}{\eps}, \frac{\text{rank}(\mA)+\log(1/\delta)}{\eps^2}\right),
\end{equation*}
for some finite number $C >0.$ Then $\mS$ suffices to be an $(1\pm \eps/4)$ embedding on $\text{Range}([\mA, \vb]) \subset \Real^{n_1n_2},$ i.e. 
\begin{equation}
\label{eps-embedding}
\pr\left(\|\mS\vy\|_2^2=(1\pm \frac{\eps}{4})\,\|\vy\|_2^2, ~~\forall ~\vy\in  \text{Range}([\mA, \vb]) \right) \geq 1-\delta.
\end{equation}
\end{lemma}

\begin{proof}[Proof of \cref{subspace embed}]
\label{subspace embed proof}
We show the proof via a standard covering net argument. In particular, based on the technique in Theorem 2.1 of \cite{W14}, by applying the sketching matrix $\mS$ on the set $\sY$ that is a $1/2$-net of $\text{Range}[\mA,\vb] \cap \S^{n_1n_2-1}$ with distortion factor $\O(\eps)$, we can achieve
\begin{equation*}
\|\mS\vy\|_2^2=(1\pm \frac{\eps}{4})\,\|\vy\|_2^2, ~~\forall ~\vy\in  \text{Range}([\mA, \vb])
\end{equation*}
Given a failure probability $\delta,$ since this covering net $\sY$ has cardinality ${\O}(9^{\text{rank}(\mA)})$ (Lemma 2.2 in \cite{W14}), \cref{JL} provides the sketching dimension
\begin{equation*}
m \geq C\,\max\left(\frac{\text{rank}^2(\mA)+\log^2(1/\delta)}{\eps}, \frac{\text{rank}(\mA)+\log(1/\delta)}{\eps^2}\right),
\end{equation*}
for some finite constant $C>0$. The proof of \cref{subspace embed} is complete. 
\end{proof}
 
Following Theorem 1 in \cite{CLNW19}, under the condition \eqref{eps-embedding} for $\mS$, the sketched solution $\hat{\vx}$ in \eqref{SQP} can be a good approximation to the true solution $\vx^*$ \eqref{QP}, i.e.,
\[
\|\mA\hat{\vx}-\vb\|_2^2 \leq (1\pm \eps)\,\|\mA\vx^*-\vb\|_2^2.
\]
The proof of \cref{unconstrained} is complete. 
\end{proof}

\subsubsection{Sparse recovery via $\ell_1$-constrained optimization}
We study the noiseless sparse recovery problem, which plays a central role in compressive sensing and signal processing. 
The goal of the recovery is to find a sparse $\hat{\vx}\in\Real^p$ to approximate an unknown signal $\overline{\vx}\in\Real^p$ (also sparse in practice), from a small number of random measurements $\mPhi \overline{\vx} \in \Real^m,$ with a short wide matrix $\mPhi\in\Real^{m \times p} (m \ll p).$ 
Furthermore, we can optimize the sparse recovery by an $\ell_1$-constrained least squares problem \cite{FNW07}, also known as the Lasso approach \cite{RT96}. Such optimization has the formulation
\begin{equation}
\label{compressed}
\hat{\vx} = \arg\min_{\|\vx\|_1\leq R} \|\mPhi\vx-\mPhi\overline{\vx}\|_2^2,
\end{equation}
where we set $R = \|\overline{\vx}\|_1>0$ to be the radius of the $\ell_1$ ball. 

We show that \cref{sketch_dimension} provides an estimation for the number of measurements $m$ to pursue an accurate recovery from the tensor sub-Gaussian sketching matrix $\mS$ \eqref{sketch}.
Recall the sketched program in \eqref{SQP} and assume $p = n_1n_2$, 
\begin{equation}
\label{sketch-lasso}
\mS =\mPhi \in\Real^{m\times p} ,\quad \mA = \id_{p} \in \Real^{p\times p}, \quad \vb =\overline{\vx}\in\Real^p,
\end{equation}
the following $\ell_1$-constrained sketched least squares, 
\begin{equation}
\label{sketch l_1}
\hat{\vx} = \arg\min_{\| \vx\|_1\leq R} \|\mS\mA\vx-\mS\vb\|_2^2
\end{equation}
is equivalent to the formulation \eqref{compressed}.

Furthermore, in the case of full data acquisition, we denote $\vx^*$ as one minimizer such that 
it has the least nonzeros among all optimizers of the following,
\begin{equation}
\label{uncompressed}
\vx^* = \arg\min_{\|\vx\|_1\leq R} \|\vx-\overline{\vx}\|_2^2\,.
\end{equation}

\begin{corollary}
\label{sparse recovery}
Let $m, p \in \mathbb{N}^+,$ $\delta \in (0, 1/2), \eps \in (0,1).$ Consider the signal to uncover $\overline{\vx}\in\Real^p$. Suppose the optimal solution $\vx^*\in\Real^p$ of \eqref{uncompressed} has $s$ nonzero entries. If $\mPhi \in \Real^{m \times p}$ in \eqref{compressed} defined as \eqref{sketch}-\eqref{sketch ii} ($n_1n_2=p$) has the number of measurements
\begin{equation}
\label{measurements}
m \geq \max\left(C\,\frac{s^2\log^2(p/s)\,\log^2(1/\delta)}{\eps^2}, 64^2\right),
\end{equation}
then with probability exceeding $1-\delta,$ the solution $\hat{\vx}\in \Real^p$ of \eqref{compressed} satisfies 
\[
\|\hat{\vx} - \overline{\vx}\|_2^2 \leq (1+\eps)^2\, \|\vx^* - \overline{\vx}\|_2^2.
\]
Above, $C >0$ is a finite constant.
\end{corollary}

We note that the number of measurement shown in \eqref{measurements} has quadratic dependence on the sparsity parameter $s$ and logarithmic factor $\log(p/s)$. Because there is a tensor structure in the testing matrix $\mPhi$, our estimation is not comparable to the sharp result of conventional sub-Gaussian measurements, where the latter case only requires $\O(s\,\log(p/s)/\eps^2)$ number of rows.

\begin{proof}[Proof of \cref{sparse recovery}]
Due to the equivalence of \eqref{compressed} and \eqref{sketch l_1} via the relationship \eqref{sketch-lasso}, we apply \cref{sketch_dimension} and obtain a bound of $m$ for successful recovery by calculating the constrained cone's $M$-complexity $M(\AKb) = M(\overline{\K}) ~(\mA = \id_p).$
We thus focus on the tangent cone based on the constraint set $\C= \{\vx\in \Real^p, ~\vert ~\|\vx\|_1\leq R\}$. Recalling the definition of $\K$ in \eqref{cone}, due to $\mathbf{\Delta} = t\, (\vx-\vx^*),$ the cone can be written as
\begin{equation*}
\K= \left\{\mathbf{\Delta}\in \Real^p~~\vert~~\|\frac{\mathbf{\Delta}}{t}+\vx^*\|_1 \leq R, ~\text{for}~t\geq 0\right\}\,.
\end{equation*}
From the definition of the $M$-complexity \eqref{M-complexity}, we apply the formula $(B.1)$ in \cite{GK20}, i.e. $\overline{\K} \cup \{{\bf 0}\}\subset \text{conv}(\sP),$ where the covering skeleton $\sP$ is a bounded $s$-sparse set $\sP:=\{\vy\in\Real^p~\vert~ \|\vy\|_0 \leq s, \|\vy\|_2 \leq 3\}.$ Furthermore, by inequalities $(B.2), (B.3)$ also in \cite{GK20}, we see the $\gamma$-functionals' bounds of $\sP$:
\begin{equation*}
\gamma_1(\sP) \leq \O\left(s\, \log(p/s)\right), \quad \gamma_2(\sP) \leq \O\left(\sqrt{s\, \log(p/s)}\right).
\end{equation*}
We then can derive an upper bound for $M(\overline{\K}):$
\begin{equation*}
M(\overline{\K}) \leq \gamma_1(\sP)+\gamma_2(\sP) \leq \O\left(s\, \log(p/s)\right).
\end{equation*}
We conclude the proof to develop the bound \eqref{measurements} by plugging the above result in \eqref{sketch_dimension_term} of \cref{sketch_dimension}.
\end{proof}

\section{Bounding the sketching size for constrained optimization}
\label{sec: 3}
In this section, we prove the sketching dimension result for constrained least squares: \cref{sketch_dimension} via \cref{sup_error}.  \cref{sup_error} gives an estimate of the supreme sketching error over an arbitrary set from the sketching matrix $\mS$.

\subsection{Analysis of tensor sub-Gaussian processes}
To bound the error \eqref{eps_distort} in \cref{sketch_dimension}, it is sufficient to bound the embedding errors for the elements in the cone $\AK.$ In this regard, we study the behavior of the (centered) distribution $\mS\vy$ with the tensor-structured design $\mS\in\Real^{m \times n_1n_2}$ \eqref{sketch}-\eqref{sketch ii} and an arbitrarily chosen vector $\vy$ from a set $\sY.$ The analysis provides the knowledge to control the embedding error and the failure probability by adjusting the number of measurements $m$, and essentially yields the sketching size estimation for convex programs. In particular, we show that the embedding error is bounded by the $M$-complexity \eqref{M-complexity} of the normalized $\sbY$ \eqref{WY} over the squared root of the sample size $m$ with high probability.

\begin{theorem}[Supreme sketching error]
\label{sup_error}
There exist universal constants $c, C>0$ for which the following holds. Let $n_1, n_2, m \in \mathbb{N}^+$ and $u \geq 64.$ For any fixed set $\sY \in \Real^{n_1n_2},$ let the sketching matrix $\mS \in \Real^{m \times n_1n_2}$ be defined in \eqref{sketch}-\eqref{sketch ii}. Then with probability at least $1-2\,\exp(-c\sqrt{m})-4\,\exp(-c\,u),$ one can achieve
\begin{equation}
\label{eqn:sup_error}
\sup_{\vy\in\sbY} \left\vert \|\mS\vy\|_2^2 - 1\right\vert \leq C\,u\,\frac{\alpha^4}{q^2}\, \frac{M(\sbY)}{\sqrt{m}}.
\end{equation}
Here, the parameters $\alpha \geq 1$ is the maximal $\psi_2$ norm of entries in $\vphi^{(1)}, \vphi^{(2)}$ and $q \in (0,1)$ is the density level, recalling the definitions in \eqref{sketch}-\eqref{sketch ii}.
\end{theorem}
The proof of \cref{sup_error} can be found in \cref{sup_error proof}. 

We stress that the failure probability term is independent of the ambient dimensions $n_1, n_2,$ which are usually large in practice. 
The error estimate \eqref{eqn:sup_error} implies that the density $q$ contributes to reducing the embedding error in a quadratic decay manner. 
However, the tensor structure in the sketching design slightly weakens the theoretical result by adding the new $M$-complexity term. Originally in the unstructured sub-Gaussian sketching setting, the $M$-complexity is substituted by the Gaussian width parameter, see Proposition 1 in \cite{PW15} for details. We show in \cref{M-complexity remark} that $M$-complexity is quantitatively bigger than Gaussian width.
We explore this numerically in \cref{sketch_size} and demonstrate that there is indeed a sketching error increase for tensor-structured sketching matrices as opposed to unstructured sketches, which validates our theory.

\subsection{Proof of \cref{sketch_dimension}}
\label{sketch_dimension proof}
\begin{proof}[Proof of \cref{sketch_dimension}]
The proof follows the ideas of Lemma 1, 2, 3 in \cite{PW15}. The next lemma relates the sketching error \eqref{eps_distort} with two specific terms $D_1$ and $D_2,$ both of which can be further bounded by the supreme embedding error in \eqref{eqn:sup_error} over the cone $\AK.$ 

\begin{lemma}{(Lemma 1 in \cite{PW15}.)}
\label{D_2/D_1}
Let $n_1, n_2, p, m\in\mathbb{N}^+$. Consider the least squares \eqref{QP} with data matrix $\mA \in \Real^{n_1n_2\times p},$ the vector $\vb\in\Real^{n_1n_2},$ the optimal solution $\vx^*\in\Real^p$ and the tangent cone $\K\subset \Real^p$ defined in \eqref{cone}. For any fixed matrix $\mS\in\Real^{m\times n_1n_2}$, define the two quantities 
\begin{equation*}
D_1 = \inf_{\vy\in \AKb} \|\mS\vy\|_2^2, \quad D_2 = \sup_{\vy\in \AKb} \left\vert\langle \frac{\mA\vx^*-\vb}{\| \mA\vx^*-\vb\|_2}, (\mS^\top\mS-\id_{n_1n_2})\,\vy \rangle\right\vert.
\end{equation*}
Then the sketched solution $\hat{\vx}$ of \eqref{SQP} satisfies 
\begin{equation}
\label{1+D_2/D_1}
\|\mA\hat{\vx}-\vb\|_2^2 \leq \left(1+\frac{2\, D_2}{D_1}\right)\,\|\mA\vx^*-\vb\|_2^2.
\end{equation}
\end{lemma}
\begin{remark}
\label{AKB=AKS}
We note that there is a difference between the set-ups of \cref{D_2/D_1} and Lemma 1 in \cite{PW15}. \cref{D_2/D_1} considers the vector $\vy$ on the normalized set $\AKb$ while \cite{PW15} defines $\vy$ to be on the intersection of $\AK$ and the unit sphere. In fact, it can be shown that $\AKb = \AK\cap \S^{n_1n_2-1},$ hence the result of \cite{PW15} is applicable for \cref{D_2/D_1}. We refer the readers to the proof of Lemma 1 in \cite{PW15} for details.
\end{remark}

It remains to bound $D_1, D_2$ to keep the accuracy of $\hat{\vx}$ for \cref{sketch_dimension}. Respectively applying \cref{sup_error}, we obtain a lower bound for $D_1$ and an upper bound for $D_2.$ 
\begin{lemma}
\label{D_1 D_2}
There exist universal constants $c, C >0$ for which the following holds. Follow the same set-up in \cref{D_2/D_1}. Fix $u \geq 64,$ then for a sketching matrix $\mS\in\Real^{m \times n_1n_2}$ defined in \eqref{sketch}-\eqref{sketch ii} and $\theta = \alpha^4/q^2$ given in \eqref{eqn:sup_error},
\begin{enumerate}
\item
the event
\begin{equation}
\label{eqn: D_1}
D_1 = \inf_{\vy \in \AKb} \|\mS\vy\|_2^2 \geq 1-C\,\theta\,u\,\frac{M(\AKb)}{\sqrt{m}}
\end{equation}
holds with probability exceeding $1-2\,\exp(-c\sqrt{m})-4\,\exp(-c\,u);$

\item
another event 
\begin{equation}
\label{eqn: D_2}
D_2 = \sup_{\vy \in \AKb} \left\vert \langle \frac{\mA\vx^* - \vb}{\|\mA\vx^* - \vb\|_2}, (\mS^\top\mS - \id_{n_1n_2}) \vy\rangle\right\vert \leq 27\,C\,\theta\,u\, \frac{M(\AKb)}{\sqrt{m}}
\end{equation}
holds with probability exceeding $1-12\,\exp(-c\sqrt{m})-24\,\exp(-c\,u).$
\end{enumerate}
\end{lemma}
The proof of \cref{D_1 D_2} is in \cref{proof: D_1 D_2}. 

We now prove \cref{sketch_dimension}. For fixed $\eps\in (0,1)$ and $\delta \in (0,1/2)$ preset in \cref{sketch_dimension}, suppose 
\begin{equation}
\label{m_assumption1}
m \geq \max(\frac{7^2\,\log^2(1/\delta)}{c^2}, 64^2)\geq \max(\frac{\log^2(42/\delta)}{c^2}, 64^2).
\end{equation}
We set 
\begin{equation*}
u = \max(\frac{7\,\log(1/\delta)}{c}, 64) \leq \sqrt{m}. 
\end{equation*}
So we have $u \geq 64.$ We now are eligible to apply \cref{D_1 D_2}. By plugging in the bounds of $D_1, D_2$ provided in \eqref{eqn: D_1}, \eqref{eqn: D_2} and summing up the failure probabilities, then together with \eqref{1+D_2/D_1} in \cref{D_2/D_1}, we have the accuracy estimation: 
\begin{equation}
\label{distortion_1}
\begin{array}{ll}
\|\mA\hat{\vx} - \vb\|_2^2&\displaystyle \leq \left(1 + \frac{54\,C\,\theta\,u\,\frac{M(\AKb)}{\sqrt{m}}}{1- C\,\theta\,u\,\frac{M(\AKb)}{\sqrt{m}}} \right)^2 \|\mA\vx^* - \vb\|_2^2\\
& \displaystyle \leq \left(1 + \frac{54\,C\,\theta\,\frac{7\,\log(1/\delta)}{c}\,\frac{M(\AKb)}{\sqrt{m}}}{1- C\,\theta\,\frac{7\,\log(1/\delta)}{c}\,\frac{M(\AKb)}{\sqrt{m}}} \right)^2 \|\mA\vx^* - \vb\|_2^2
\end{array}
\end{equation}
holds with probability exceeding 
\[
1 - 14\, \exp(-c\, \sqrt{m}) - 28\, \exp(-c\, u) \geq 1- 42\,\exp(-c\, u) \geq 1 - \delta.
\]
The above inequalities are derived by $u \leq \sqrt{m}$ and \eqref{m_assumption1}. 

It suffices to further assume that 
\begin{equation}
\label{m_assumption2}
m \geq \max\left(\frac{55^2\, C^2}{c^2}\, \theta^2\,\frac{7^2\,\log^2(1/\delta)\,M(\AKb)^2}{\eps^2}, \frac{7^2\,\log^2(1/\delta)}{c^2}, 64^2\right).
\end{equation}
Then by \eqref{distortion_1}, one can achieve
\begin{equation*}
\|\mA\hat{\vx} - \vb\|_2^2\displaystyle \leq  \left(1+\frac{\frac{54\,\eps}{55}}{1-\frac{\eps}{55}}\right)^2  \|\mA\vx^* - \vb\|_2^2\leq  \left(1+\eps\right)^2  \|\mA\vx^* - \vb\|_2^2
\end{equation*}
with probability at least $1-\delta.$

We define a finite constant 
\begin{equation}
\label{C}
\tilde{C} = \max\left( \frac{55^2\cdot7^2\,C^2}{c^2}\, \theta^2, \frac{7^2}{c^2}\right),
\end{equation}
where $\theta = \alpha^4/q^2 >1$ given in \eqref{eqn:sup_error} is a finite number and only depends on the distribution of the sketching matrix $\mS$. Then the assumption \eqref{m_assumption2} on $m$ can be simplified as follows
\begin{equation*}
m \geq \displaystyle \max\left(\tilde{C}\,\frac{\log^2(1/\delta)\,M(\AKb)^2}{\eps^2},  64^2\right),
\end{equation*} 
which ensures the sketched solution $\hat{\vx}$ achieves \eqref{eps_distort} with probability exceeding $1-\delta.$ We replace $\tilde{C}$ with the new notation $C$ and conclude the proof of \cref{sketch_dimension}. 
\end{proof}

\subsection{Proof of \cref{D_1 D_2}}
\label{proof: D_1 D_2}
\begin{enumerate}
\item
\begin{proof}[Estimation of $D_1$]
By letting $\sY = \AK,$ \cref{sup_error} implies that with probability at least $1-2\,\exp(-c\sqrt{m}) - 4\,\exp(-c\,u),$
\begin{equation*}
1-\inf_{\vy\in\AKb} \|\mS\vy\|_2^2 \leq \sup_{\vy\in\AKb} \left\vert \|\mS\vy\|_2^2-1\right\vert \leq C\,u\,\frac{\alpha^4}{q^2}\, \frac{M(\AKb)}{\sqrt{m}}.
\end{equation*}
We complete the proof for $D_1$ by adjusting the above inequality to \eqref{eqn: D_1}.
\end{proof}

\item
\begin{proof}[Estimation of $D_2$]
The proof follows the same logic as Lemma 3 in \cite{PW15}. 
We use short-hand notations: $\mC = \mS^\top\mS - \id_{n_1n_2}\in\Real^{n_1n_2\times n_1n_2}$ and $\vz = \frac{\mA\vx^*-\vb}{\|\mA\vx^*-\vb\|_2} \in \S^{n_1n_2-1}$. Consider two subsets of $\AKb\subset\S^{n_1n_2-1}:$
\begin{equation}
\label{AK_part}
 \AKb_+ = \{\vy\in \AKb~~\vert~~ \langle \vz, \vy \rangle \geq 0\}, \quad  \AKb_- = \{\vy\in \AKb~~\vert~~ \langle \vz, \vy \rangle < 0\}.
\end{equation}
Then
\begin{equation}
\label{D_2_part}
D_2 = \sup_{\vy\in \AKb} \left\vert\vy^\top \mC\vz \right\vert = \max\left(\sup_{\vy\in \AKb_+}|\vy^\top\mC\vz|, ~~ \sup_{\vy\in \AKb_-}|\vy^\top\mC\vz|\right).
\end{equation}
In the case of $\sup_{\vy\in \AKb+}|\vy^\top\mC\vz|,$ we show its partition
\begin{equation}
\label{+_part}
\sup_{\vy\in \AKb_+}|\vy^\top\mC\vz| \leq \frac{1}{2}\, \left( \sup_{\vy\in \AKb_+}\vert(\vy+\vz)^\top\mC (\vy+\vz)\vert+ \sup_{\vy\in \AKb_+} \vert\vy^\top\mC\vy\vert+ \sup\vert\vz^\top\mC\vz\vert\right).
\end{equation}
We then apply \cref{sup_error} three times to respectively bound each term on the right-hand side of \eqref{+_part}. Taking a union of the failure of each term in \eqref{+_part}, we know that the following event 
\begin{equation}
\label{part_error}
\begin{array}{l}
\displaystyle\sup_{\vy\in \AKb_+}\left\vert(\frac{\vy+\vz}{\|\vy+\vz\|_2})^\top\mC \frac{\vy+\vz}{\|\vy+\vz\|_2}\right\vert\displaystyle \leq C\,u\,\frac{\alpha^4}{q^2}\, \frac{M(\overline{\AKb_+ +\vz})}{\sqrt{m}} \\\\
\displaystyle\sup_{\vy\in \AKb_+} \left\vert\vy^\top\mC\vy \right\vert \displaystyle \leq C\,u\,\frac{\alpha^4}{q^2}\, \frac{M(\AKb_+)}{\sqrt{m}} \\\\
\displaystyle\sup\left\vert\vz^\top\mC\vz \right\vert \displaystyle \leq C\,u\,\frac{\alpha^4}{q^2}\, \frac{M(\{\vz\})}{\sqrt{m}}.
\end{array}
\end{equation}
simultaneously holds with probability exceeding $1-6\,\exp(-c\sqrt{m})-12\,\exp(-c\,u).$ 

The first inequality of \eqref{part_error} implies for $\vy, \vz \in \S^{n_1n_2-1,}$ 
\begin{equation}
\label{y+z norm}
\begin{array}{ll}
\displaystyle\sup_{\vy\in \AKb_+}\vert(\vy+\vz)^\top \mC(\vy+\vz)\vert &\displaystyle\leq \left(\sup_{\vy\in \AKb_+} \|\vy+\vz\|_2^2\right)\,C\,u\,\frac{\alpha^4}{q^2}\, \frac{M(\overline{\AKb_+ +\vz})}{\sqrt{m}} \\
&\displaystyle\leq 4\,C\,u\,\frac{\alpha^4}{q^2}\, \frac{M(\overline{\AKb_+ +\vz})}{\sqrt{m}}.
\end{array}
\end{equation}

To develop a bound for $\sup_{\vy\in \AKb_+}|\vy^\top\mC\vz|$ by \eqref{+_part}, \eqref{part_error}, in the remaining proofs, we aim to estimate $M(\overline{\AKb_+ +\vz})$ and $M(\{\vz\})$ by the term $M(\AKb_+).$ 

We first claim that $M(\overline{\AKb_+ +\vz})$ is no greater than $M(\AKb_+)$ up to a constant. 
\begin{lemma}
\label{lemma: M(AK+z)}
\begin{equation}
\label{M(AK+z)}
M(\overline{\AKb_++\vz}) \leq 13\, M(\AKb_+). 
\end{equation}
\end{lemma}
The proof of \cref{lemma: M(AK+z)} is in \cref{proof: M(AK+z)}.

For a single point set $\{\vz\}\subset \S^{n_1 n_2-1}$, it can be shown that
\begin{equation}
\label{M(z)} 
M(\{\vz\}) \leq 1 \leq M(\AKb_+).
\end{equation}

Combine the results of \eqref{+_part} - \eqref{M(z)}, we know the event
\begin{equation}
\begin{array}{ll}
\displaystyle \sup_{\vy\in\AKb_+}\vert\vy^\top\mC\vz\vert& \leq \displaystyle \frac{1}{2}\,\left(4\,C\,u\,\frac{\alpha^4}{q^2} \,\frac{M(\overline{\AKb_++\vz})}{\sqrt{m}}+C\,u\,\frac{\alpha^4}{q^2} \,\frac{M(\AKb_+)}{\sqrt{m}}+C\,u\,\frac{\alpha^4}{q^2} \,\frac{M(\{\vz\})}{\sqrt{m}}\right)\\
& \displaystyle \leq \frac{1}{2}\,(4\times 13+1+1)\,C\,u\,\frac{\alpha^4}{q^2} \,\frac{M(\AKb_+)}{\sqrt{m}} \leq 27\,C\,u\,\frac{\alpha^4}{q^2} \,\frac{M(\AKb)}{\sqrt{m}}
\end{array}
\end{equation}
holds with probability at least $1-6\,\exp(-c\sqrt{m})-12\,\exp(-cu)$.

We can establish the same bound for $\sup_{\vy\in \AKb_-}|\vy^\top\mM\vz|$ following a similar argument for $\sup_{\vy\in \AKb_+}|\vy^\top\mM\vz|.$ Taking a union of the failure of each term in \eqref{D_2_part} being well controlled, we can conclude the proof by with probability exceeding $1-12\,\exp(-c\sqrt{ m}) -24 \,\exp(-c\, u),$
\begin{equation*}
D_2 = \sup_{\vy\in \AKb} \left\vert\vy^\top \mC\vz \right\vert \leq 27\,C\,u\,\frac{\alpha^4}{q^2}\,\frac{M(\AKb)}{\sqrt{m}}.
\end{equation*}
The proof for $D_2$ is finished.
\end{proof}
\end{enumerate}

\section{Bounding the sketching error} 
\label{sec: 4}
In \cref{sec: 4}, we present the proofs of the JL property \cref{JL} and the supreme sketching error result \cref{sup_error}. Our proofs boil down to quantifying the sketching error for the tensor-structured sketching matrices. 
\subsection{Concentration properties} 
Recalling the definition of $\mS \in \Real^{m \times n_1n_2}$ in \eqref{sketch}-\eqref{sketch ii}, for a fixed vector $\vy\in\Real^{n_1n_2}$, the target sketching error $\|\mS\vy\|_2^2 - \|\vy\|_2^2$ can be expressed as:
\begin{equation}
\label{eqn: err express}
\|\mS\vy\|_2^2-\|\vy\|_2^2 = \frac{1}{m}\, \sum_{k=1}^m \langle \veta_k\otimes \bxi_k, \vy \rangle^2 - \E\, \langle \veta_k\otimes \bxi_k, \vy \rangle^2.
\end{equation}
In the aim of evaluating the quadratic error term \eqref{eqn: err express}, in \cref{linear}, we first lay out basic concentration properties of the linear sketching $\langle \veta \otimes \bxi, \vy\rangle$, where $\veta \otimes \bxi$ is a single i.i.d. tensorized sub-Gaussian process in $\mS$. 
Because of the tensor product structure in this random form, \cref{linear} shows the concentration of $\langle \veta \otimes \bxi, \vy\rangle$ exhibits the sub-exponential behavior which is defined as follows.
\begin{definition}
\label{psi_1}
The sub-exponential norm of a random variable $x\in\Real$, denoted by $\|x\|_{\psi_1}$, is defined as:
\[
\|x\|_{\psi_1} = \inf\{t>0: \quad \E \,\exp(x/t) \leq 2\}.
\]
A random variable is called sub-exponential if it has a bounded $\psi_1$ norm.
\end{definition}
We present the exact result of \cref{linear}.
\begin{proposition}
\label{linear}
 There exist universal constants $c, C>0$ for which the following holds. Let $n_1, n_2 \in \mathbb{N}^+$ and $q\in(0,1]$. For any vector $\vy \in \sbY\subset \S^{n_1n_2-1},$ draw independent vectors $\veta \in \Real^{n_1}, \bxi\in\Real^{n_2}$ defined in \eqref{sketch ii} constructed with density level $q$. Then the random variable $\langle \veta \otimes \bxi, \vy\rangle$ has the following properties:
\begin{equation}
\E \,\langle \veta \otimes \bxi, \vy\rangle = 0, ~~ \E \,\langle \veta \otimes \bxi, \vy\rangle^2 = 1, ~~ \|\langle \veta \otimes \bxi, \vy\rangle\|_{\psi_1} \leq C\, \frac{\alpha^2}{q}.
\end{equation}
Moreover, for any $t>0$, it satisfies 
\begin{equation}
\label{linear concentration}
\pr\left(\vert\langle \veta \otimes \bxi, \vy\rangle\vert \geq t\right) \leq 2\,\exp\left(-c\, \min \left(\frac{t^2}{\frac{\alpha^4}{q^2}\,\|\vy\|_2^2}, \frac{t}{\frac{\alpha^2}{q}\,\|\vy\|_2}\right)\right).
\end{equation}
Here, $\alpha \geq 1$ is the maximal $\psi_2$ norm of $\vphi^{(1)}, \vphi^{(2)}$, recalling definition in \eqref{sketch ii}.
\end{proposition}
We next focus on generalizing the concentration analysis to the quadratic form $\|\mS\vy\|_2^2 - \|\vy\|_2^2$ \eqref{eqn: err express}. However, the sub-exponential property of the linear form $\langle \veta \otimes \bxi, \vy\rangle$ would bring more difficulties for studying the quadratic concentration. 

\subsection{Proof of \cref{JL}}
In this subsection, we prove the JL property of the tensorized sub-Gaussian sketches: \cref{JL}.
\label{JL proof}
\begin{proof}[Proof of \cref{JL}]
Without the loss of generality, we assume the subset $\sY$ is on the unit sphere $\S^{n_1n_2-1}$. For any fixed vector $\vy\in\sY,$ since $\langle \veta_k\otimes \bxi_k, \vy \rangle$ for $k \in [m]$ are independent copies of $ \langle \veta\otimes \bxi, \vy \rangle,$ by \cref{linear},
\begin{equation*}
\E \,\langle \veta_k\otimes \bxi_k, \vy \rangle = 0, ~~\E\, \langle \veta_k\otimes \bxi_k, \vy \rangle^2 = 1, ~~\|\langle \veta_k\otimes \bxi_k, \vy \rangle\|_{\psi_1} \leq C\, \frac{\alpha^2}{q}.
\end{equation*}

To estimate the sketching error \eqref{eqn: err express}, we introduce the following concentration inequality for quadratic forms involving sub-exponential random variables from \cite{GSS21}.
\begin{proposition}[Proposition 1.1, $\alpha=1$ case in \cite{GSS21}]
\label{quadratic}
There exists a universal constant $c>0$ for which the following holds. Let $\mA\in\Real^{m \times m}$ be a symmetric matrix. Suppose $x_1, \dots, x_k,\dots, x_m$ be independent random variables satisfying $\E \,x_k = 0, \E \,x_k^2 = \sigma_k^2, \|x_k\|_{\psi_1} \leq K,$ for $k \in [m]$. For any $t>0,$ we have 
\begin{equation}
\pr\left(\left\vert\sum_{k,\ell=1}^m a_{k,\ell} \,x_k\,x_\ell - \sum_{k=1}^m \sigma_k^2 a_{k,k}\right\vert\geq t \right) \leq 2\, \exp\left(-c\, \min \left(\left(\frac{t}{K^2\,\|\mA\|_2}\right)^{1/2}, \frac{t^2}{K^4\,\|\mA\|_F^2}\right)\right).
\end{equation}
\end{proposition}
We apply \cref{quadratic} by setting 
\begin{equation*}
\mA = \frac{1}{m} \, \id_m \in \Real^{m \times m}, ~~x_k = \langle \veta_k\otimes \bxi_k, \vy \rangle, ~~\sigma_k = 1, ~~K =  C\, \frac{\alpha^2}{q},~~ t =\eps.
\end{equation*}
Then we have
\begin{equation*}
\|\mA\|_2 = \frac{1}{m},~~\|\mA\|_F = \frac{1}{\sqrt{m}},
\end{equation*}
and further by \eqref{eqn: err express}, for a fixed $\vy\in\sY$,
\begin{equation*}
\begin{array}{ll}
\displaystyle\pr\left(\left\vert \|\mS\vy\|_2^2-\|\vy\|_2^2   \right\vert \geq\eps\right)
& \displaystyle=  \pr\left(\left\vert \sum_{k=1}^m\frac{x_k^2}{m} -\sum_{k=1}^m\frac{1}{m}\right\vert\geq \eps\right)\\
& \displaystyle\leq 2\, \exp\left(-c\, \min\left(\left(\frac{\eps}{C^2\,\frac{\alpha^4}{q^2}\,\frac{1}{m}}\right)^{1/2}, \frac{\eps^2}{C^4\,\frac{\alpha^8}{q^4}\,\frac{1}{m}}\right)\right)\\
& \displaystyle\leq 2\,\exp\left(-c\,\min\left(\left(\frac{m\,\eps}{C^2\,\frac{\alpha^4}{q^2}}\right)^{1/2}, \frac{m\,\eps^2}{C^4\,\frac{\alpha^8}{q^4}} \right)\right).
\end{array}
\end{equation*}
Taking a union bound of the above event for all vectors $\vy\in\sY,$ we have
\begin{equation}
\label{union bound 1}
\pr\left(\left\vert \|\mS\vy\|_2^2-\|\vy\|_2^2   \right\vert \geq\eps, ~~\exists \,\vy\in \sY\right)\leq 2\,|\sY|\,\exp\left(-c\,\min\left(\left(\frac{m\,\eps}{C^2\,\frac{\alpha^4}{q^2}}\right)^{1/2}, \frac{m\,\eps^2}{C^4\,\frac{\alpha^8}{q^4}} \right)\right).
\end{equation}

Given the embedding dimension condition $m$ in \eqref{jl m}: 
\begin{equation}
\label{jl m ii}
m \geq C'\, \max\left(\frac{\log^2(2\,|\sY|/\delta)}{\eps}, \frac{\log(2\,|\sY|/\delta)}{\eps^2}\right),
\end{equation}
\footnote{We replace constant $C$ in the original formula with $C'$ to avoid confusion.} where we we set the constant $C'$ to satisfy
\begin{equation}
\label{C jl}
C' \geq \max\left(\frac{C^2\,\frac{\alpha^4}{q^2}}{c^2}, \frac{C^4\, \frac{\alpha^8}{q^4}}{c}\right).
\end{equation}
Combining the results in \eqref{union bound 1}, \eqref{jl m ii} and \eqref{C jl}, we obtain 
\begin{equation*}
\begin{array}{ll}
\pr\left(\left\vert \|\mS\vy\|_2^2-\|\vy\|_2^2   \right\vert \geq\eps, ~~\exists \,\vy\in \sY\right)&\displaystyle  \leq 2\,|\sY|\,\exp\left(-\min\left(c\,\frac{\sqrt{C'}\,\log(2|\sY|/\delta)}{C\,\frac{\alpha^2}{q}}, c\,\frac{C'\,\log(2|\sY|/\delta)}{C^4\,\frac{\alpha^8}{q^4}}\right) \right)\\
& \leq 2\,|\sY|\, \exp(-\log(2|\sY|/\delta)) = \delta.
\end{array}
\end{equation*}
The proof of \cref{JL} is complete.
\end{proof}

\subsection{Proof of \cref{sup_error}}
\label{sup_error proof}
In this subsection, we prove \cref{sup_error} and demonstrate how to bound the supreme sketching error of the tensor-structured sketching matrix $\mS$ over an arbitrary set.

\begin{proof}[Proof of \cref{sup_error}]
As shown in \cref{linear} that the linear form of our sketching has a sub-exponential concentration property, the proof of \cref{sup_error} for bounding the quadratic sketching error draws on a multiplier form estimation in regard sub-exponential random variables from \cite{GK20}. We present this useful result in \cref{multiplier}. In essence, we utilize \cref{multiplier} and show that the $M$-complexity (\cref{M def}) fits to bound the tensorized sub-Gaussian processes approximation error. 

To begin with, we define a generic Bernstein inequality that characterizes sub-exponential concentrations. 
\begin{definition}[Generic Bernstein concentration \cite{GK20}]
\label{Bernstein}
Let $\vx \in \Real^n$ be a random vector and $\|\cdot\|_{\text{g}}, \|\cdot\|_{\text{e}}$ be two semi-norms on $\Real^n$. We say $\vx$ satisfies generic Bernstein concentration if for every $\vy \in \Real^n$ and every $t \geq 0,$ we have 
\begin{equation}
\pr\left(|\langle \vx, \vy\rangle| \geq t\right) \leq 2\,\exp\left(-\min \left(\frac{t^2}{\|\vy\|^2_{\text{g}},} \frac{t}{\|\vy\|_{\text{e}}}\right)\right).
\end{equation}
\end{definition}
The following \cref{multiplier} states that the $M$-complexity is tailored for a multiplier estimation of random vectors satisfying certain generic Bernstein condition. 
\begin{proposition}[Proposition 5.15 in \cite{GK20}]
\label{multiplier}
There exist universal constants $c, C>0$ for which the following holds. Let $\sbY \subset \S^{n-1}$ and $(\vx, w) \in \Real^{n}\times \Real$ be a random pair such that $\|w\|_{\psi_1} \leq K$ and $\vx$ satisfies generic Bernstein concentration with respect to $(\|\cdot\|_{\text{g}}, \|\cdot\|_{\text{e}})$.  For every $u\geq64,$ suppose $(\vx_1, w_1), \dots, (\vx_m, w_m)$ are independent copies of $(\vx, w).$ Then the following holds true with probability at least $1 - 2\,\exp(-c\, \sqrt{m}) - 4\,\exp(-c\, u),$
\begin{equation}
\sup_{\vy\in\sbY}\left\vert \frac{1}{m}\,\sum_{k=1}^m w_k\, \langle \vx_k, \vy \rangle - \E\,w\,\langle \vx, \vy \rangle \right\vert \leq C\, u\,K\, \frac{M^{\text{(g,e)}}(\sbY)}{\sqrt{m}}.
\end{equation}
\end{proposition}

To transform the notations of \cref{multiplier} to our setting, we set $n = n_1n_2,$ the random pair $(\vx, w) \in \Real^{n_1n_2} \times \Real$ to be
\begin{equation}
\label{x,w}
\vx = \veta\otimes \bxi, \quad w = \langle \vx, \vy\rangle = \langle \veta\otimes \bxi, \vy \rangle.
\end{equation}
From \cref{linear}, we know  
\begin{equation}
\label{K}
\E\,w\,\langle \vx, \vy \rangle = \E\,\langle \veta\otimes \bxi, \vy \rangle^2 = 1, \quad \|w\|_{\psi_1} = \|\langle \veta\otimes \bxi, \vy \rangle\|_{\psi_1} \leq K=\frac{\alpha^2}{q},
\end{equation}
and $\vx = \veta \otimes \bxi$ exhibits generic Bernstein concentration w.r.t. $(\|\cdot\|_{\text{g}}, \|\cdot\|_{\text{e}}) = (\frac{\alpha^2}{\sqrt{c}\,q}\,\|\cdot\|_2, \frac{\alpha^2}{c\,q}\,\|\cdot\|_2)$ based on \cref{Bernstein}.

Hence, for $u \geq 64$, we are eligible to apply \cref{multiplier} with the parameters relations set in \eqref{x,w}, \eqref{K}. We can obtain that, with probability exceeding $1 - 2\,\exp(-c\, \sqrt{m}) - 4\,\exp(-c\, u),$
\begin{equation*}
\begin{array}{ll}
\displaystyle \sup_{\vy\in\sbY} \left\vert \|\mS\vy\|_2^2 - 1\right\vert  &\displaystyle= \sup_{\vy\in\sbY}\left\vert \frac{1}{m}\,\sum_{k=1}^m w_k\, \langle \vx_k, \vy \rangle - \E\, w\,\langle \vx, \vy \rangle \right\vert  \\
& \displaystyle\leq C\, u\,\frac{\alpha^2}{q} \,\frac{M^{\text{(g,e)}}(\sbY)}{\sqrt{m}} \leq \frac{C}{\min(\sqrt{c}, c)}\, u\,\frac{\alpha^4}{q^2} \,\frac{M(\sbY)}{\sqrt{m}}.
\end{array}
\end{equation*} 
For the last equality, it can be easily shown that $M^{\text{(g,e)}}$ has a linear relationship with $M^{(2,2)} \equiv M,$ i.e. $M^{\text{(g,e)}}(\sbY) \leq \frac{\alpha^2}{\min(\sqrt{c}, c)\, q}\, M(\sbY)$ for semi-norms $(\|\cdot\|_{\text{g}}, \|\cdot\|_{\text{e}}) = (\frac{\alpha^2}{\sqrt{c}}\,q\,\|\cdot\|_2, \frac{\alpha^2}{c\,q}\,\|\cdot\|_2)$.
 
We replace $\frac{C}{\min(\sqrt{c}, c)}$ in the above inequality with a new notation $C$ and conclude the proof of \cref{sup_error}.
\end{proof}

\subsection{Proof of \cref{linear}}
\label{proof:linear}
Let us first lay out some auxiliary results. 
\begin{definition}
For $p \geq 1$ and a random variable $x$,  the $L_p$ norm is defined as 
\begin{equation}
\label{lp}
\|x\|_p = \left(\E |x|^p \right)^{1/p}.
\end{equation}
\end{definition}

\begin{lemma}[Exercise 6.3.5 in \cite{V18}]
\label{symmetrization}
Let $F: \Real^+ \to \Real$ be an increasing, convex function and $x_1, \dots, x_n$ be independent, zero-mean random variables in a normed space, then
\begin{equation}
\E \,F\left(\left\|\sum_{i=1}^n x_i\right\|\right) \leq \E\, F\left(2\,\left\|\sum_{i=1}^n r_i\, x_i\right\|\right),
\end{equation}
where $r_1, \dots, r_n \in \{-1,1\}$ are i.i.d. Rademacher variables. 
\end{lemma}

\begin{lemma}[Lemma 5.4, $k=1$ case in \cite{AW13}]
\label{subgaussian_to_gaussian}
There exists a universal constant $C>0$ for which the following holds. For any $p \geq 2,$ and $a_1, \dots, a_n \in \Real,$ if $x_1, \dots, x_n$ are independent symmetric random variables with $\|x_i\|_{\psi_2} \leq \beta,$ then 
\begin{equation}
\left\|\sum_{i=1}^n a_i\,x_i\right\|_p \leq C\,\beta\, \left\|\sum_{i=1}^n a_i\, g_i\right\|_p,
\end{equation}
where $g_1, \dots, g_n \in \Real$ are i.i.d. $\nor(0,1)$ variables. 
\end{lemma}

\begin{lemma}[Theorem 1, $d=2$ case in \cite{L06}]
\label{gaussian bound}
There exists a universal constant $C>0$ for which the following holds. For any $p \geq 2$ and any vector $\vy \in \Real^{n_1n_2},$ draw independent standard normal vectors $\vg_1 \in \Real^{n_1}$ and $\vg_2 \in \Real^{n_2}.$ Then 
\begin{equation}
\label{eqn: gaussian bound}
\|\langle \vg_1 \otimes \vg_2, \vy \rangle \|_p \leq C\, (\sqrt{p}\, \|\mY\|_F+p\, \|\mY\|_2),
\end{equation}
where the matrix $\mY \in \Real^{n_1 \times n_2}$ is a reshaped form of the vector $\vy\in\Real^{n_1n_2},$ i.e. $Y_{i_1, i_2} = y_{i_1(n_1 -1)+i_2}$ for $i_1 \in [n_1], i_2 \in [n_2].$
\end{lemma}

We now prove \cref{linear}. 
\begin{proof}[Proof of \cref{linear}]
By definitions of $\veta$ and $\bxi$ in \eqref{sketch ii}, their entries are independent, mean-zero and unit-variant random variables. Thus the first and second order moments of $\langle \veta\otimes \bxi, \vy\rangle$ are
\begin{equation*}
\E\, \langle \veta\otimes \bxi, \vy\rangle = \E \,\veta^\top\mY\bxi = \E \sum_{i_1=1}^{n_1} \left(\sum_{i_2=1}^{n_2}Y_{i_1, i_2} \xi_{i_2} \right)\, \eta_{i_1} = 0.
\end{equation*}

\begin{equation*}
\begin{array}{ll}
\E\, \langle \veta\otimes \bxi, \vy\rangle^2 &\displaystyle= \E \sum_{i_1,j_1=1}^{n_1} \sum_{i_2, j_2=1}^{n_2}(Y_{i_1, i_2} \xi_{i_2} \eta_{i_1} )\,(Y_{j_1, j_2} \xi_{j_2} \eta_{j_1} )\\\\
&\displaystyle=\E \sum_{i_1=1}^{n_1} \sum_{i_2=1}^{n_2}Y_{i_1, i_2}^2 \xi_{i_2}^2 \eta_{i_1}^2 = \sum_{i_1=1}^{n_1} \sum_{i_2=1}^{n_2}Y_{i_1, i_2}^2 = \|\vy\|_2^2 = 1.
\end{array}
\end{equation*}
Here, the matrix $\mY \in \Real^{n_1 \times n_2}$ is a reshaped form of the vector $\vy\in\Real^{n_1n_2}.$

For any $p \geq 2,$ let the convex, increasing function $F(|x|) = |x|^p,$ thus $(\E\,F(|x|) )^{1/p}= \|x\|_p$. We apply the symmetrization inequality \cref{symmetrization} twice, respectively to transform $\veta$ and $\bxi$ into symmetric random vectors,
\begin{equation}
\label{symmetrization1}
\begin{array}{ll}
\|\langle \veta \otimes \bxi, \vy\rangle\|_p &\displaystyle= \|\veta^\top\mY\bxi\|_p = \left\|\sum_{i_2=1}^{n_2}\sum_{i_1=1}^{n_1}Y_{i_1, i_2} \eta_{i_1}\xi_{i_2} \right\|_p\\
&\displaystyle\leq  2\cdot 2 \,\left\| \sum_{i_2=1}^{n_2}\sum_{i_1=1}^{n_1} Y_{i_1, i_2}\left((r_1)_{i_1} \eta_{i_1} \right)\,\left((r_2)_{i_2} \xi_{i_2} \right)\right\|_p\\
&\displaystyle = 4\, \left\|(\vr_1 \circ \veta)^\top\, \mY\, (\vr_2 \circ \bxi) \right\|_p = 4\, \left\|\langle (\vr_1 \circ \veta)\otimes (\vr_2 \circ \bxi), \vy\rangle \right\|_p,
\end{array}
\end{equation}
where $\vr_1 = [(r_1)_1, \dots, (r_1)_{n_1}]^\top \in \{-1,1\}^{n_1}$ and $\vr_2 = [(r_2)_1, \dots, (r_2)_{n_2}]^\top \in \{-1,1\}^{n_2}$ are independent Rademacher vectors.

Since the entries of $\veta, \bxi$ satisfy $\|\eta_{i_1}\|_{\psi_2}, \|\xi_{i_2}\|_{\psi_2} \leq \alpha/\sqrt{q}$, recalling definition in \eqref{sketch ii}, by Exercise 6.3.6 in \cite{V18}, the entries in $\vr_1 \circ \veta$ and $\vr_2 \circ \bxi$ have bounded $\psi_2$ norm by $\alpha/\sqrt{q}$ up to a constant $C$.  Also because the entries of $\vr_1 \circ \veta$ and $\vr_2 \circ \bxi$ are independent and symmetric, we apply \cref{subgaussian_to_gaussian} twice by setting $\beta = C\alpha/\sqrt{q}$ and transform $\vr_1 \circ \veta$ and $\vr_2 \circ \bxi$ to Gaussian vectors,
\begin{equation}
\label{subgaussian_to_gaussian2}
\begin{array}{ll}
\displaystyle\left\|\langle (\vr_1 \circ \veta) \otimes (\vr_2 \circ \bxi), \vy\rangle \right\|_p
&\displaystyle= \left\|\sum_{i_2=1}^{n_2}\sum_{i_1=1}^{n_1}Y_{i_1,i_2}\left((r_1)_{i_1}\xi_{i_1}\right)\, \left((r_2)_{i_2}\eta_{i_2}\right)\right\|_p\\
&\displaystyle\leq C^2\, \frac{\alpha^2}{q}\,\left\|\sum_{i_2=1}^{n_2}\sum_{i_1=1}^{n_1}Y_{i_1,i_2}(g_1)_{i_1} (g_2)_{i_2}\right\|_p
\displaystyle = C^2\, \frac{\alpha^2}{q}\,\left\|\langle \vg_1\otimes \vg_2,\vy\rangle\right\|_p.
\end{array}
\end{equation}

Combining the results of \eqref{symmetrization1},\eqref{subgaussian_to_gaussian2} and \eqref{eqn: gaussian bound} in \cref{gaussian bound}, we have 
\begin{equation}
\label{lp bound}
\begin{array}{ll}
 \|\langle \veta \otimes \bxi, \vy\rangle\|_p &\leq 4\, \left\|\langle (\vr_1 \circ \veta)\otimes (\vr_2 \circ \bxi), \vy\rangle \right\|_p \\
&\displaystyle \leq 4C^2\, \frac{\alpha^2}{q}\,\left\|\langle \vg_1\otimes \vg_2,\vy\rangle\right\|_p \\
&\displaystyle\leq C\, \frac{\alpha^2}{q}\,(\sqrt{p}\, \|\mY\|_F+p\, \|\mY\|_2).
\end{array}
\end{equation}

The above term can be further bounded by $C\,\frac{\alpha^2}{q}\,p$ for $p \geq 2$ because $\|\mY\|_2\leq \|\mY\|_F =\|\vy\|_2 =1$. By Proposition 2.7.1 of \cite{V18}, it implies there exists a universal constant $C>0$ such that
\begin{equation*}
\E\, \exp\left(\frac{| \langle \veta \otimes \bxi, \vy\rangle|}{C\,\frac{\alpha^2}{q}\,p}\right) \leq 2,
\end{equation*}
which is equivalent to 
$\|\langle \veta \otimes \bxi, \vy\rangle\|_{\psi_1} \leq C\, \frac{\alpha^2}{q}
$
by \cref{psi_1}.

In order to prove $\langle \veta \otimes \bxi, \vy\rangle$ satisfies the concentration \eqref{linear concentration}, we rewrite the result of \eqref{lp bound},
\[
 \|\langle \veta \otimes \bxi, \vy\rangle\|_p \leq \max \left(C\, \frac{\alpha^2}{q}\, \|\mY\|_F\,\sqrt{p} ,\quad C\, \frac{\alpha^2}{q}\, \|\mY\|_2\,p \right).
\]
Based on Proposition 2.5.2 and 2.7.1 from \cite{V18}, the above two bounds in max function imply there exists a universal constant $c>0$ such that
\begin{equation*}
\begin{array}{ll}
\pr\left(\vert \langle \veta \otimes \bxi, \vy\rangle\vert \geq t \right)& \leq \displaystyle 2\, \exp\left( - c\,\min\left( \frac{t^2}{ \frac{\alpha^4}{q^2}\,\|\mY\|_F^2},\frac{t}{ \frac{\alpha^2}{q}\,\|\mY\|_2} \right)\right)\\
&\displaystyle \leq 2\, \exp\left( - c\,\min\left( \frac{t^2}{ \frac{\alpha^4}{q^2}\,\|\vy\|_2^2},\frac{t}{ \frac{\alpha^2}{q}\,\|\vy\|_2} \right)\right).
\end{array}
\end{equation*}

The proof of \cref{linear} is complete.
\end{proof}

\section{Numerical experiments}
\label{sec: numerics}
In this section, we show the numerical performance for the row-wise tensor sketching matrix $\mS\in\Real^{m \times n_1n_2}$ defined in \eqref{sketch}-\eqref{sketch ii}. Specifically, we choose a common case of $\mS$ to be: Gaussian+Rademacher (G+R), meaning the tensor component $\vphi^{(1)}\in\Real^{n_1}$ in \eqref{sketch ii} follows the standard normal distribution $\nor(\mathbf{0}, \id_{n_1}),$ while the other $\vphi^{(2)}\in\{-1,1\}^{n_2}$ is a Rademacher vector whose entries are i.i.d. and take the values $1$ and $-1$ with equal probability $1/2.$

We implement the sketch $\mS$ on the unconstrained linear regression
\begin{equation}
\label{unconstrained_numerics}
\vx^* = \arg\min_{\vx\in\Real^p}\|\mA\vx-\vb\|_2^2.
\end{equation}
We focus on three types of linear regressions where the data matrix $\mA\in\Real^{n_1n_2\times p}$ has different properties: 
\begin{enumerate}
\item {\bf{Well-conditioned program.}}\\
In this setting, the data matrix $\mA$ has a relatively small condition number.

\item {\textbf{Ill-conditioned program.}}\\
For the ill-conditioned problem, we build the condition number of $\mA$ to be large, for instance, $10^4.$

\item {\bf{Structured program.}}\\
As we mentioned in \cref{design}, in the structured program, the columns of the matrix $\mA$ admit tensor structure:
\begin{equation}
\label{structured}
\va_j = \vf_j\otimes \vg_j, \quad\text{for~} j \in [p].
\end{equation}
\end{enumerate}

In terms of building the input system, we generate the matrix $\mA$ based on the singular value decomposition:
\begin{equation*}
\underset{n_1n_2\times p}{\mA}\quad = \underset{n_1n_2\times p}{\mU}\quad\underset{p\times p}{\mSig}\quad\underset{p \times p}{\mV^\top},
\end{equation*}
where the left and right factors $\mU, \mV$ are built from the orthogonal matrices $\mQ$ in QR decompositions of standard Gaussian matrices, whose entries are i.i.d. and normally distributed with zero mean and unit variance.
For a well-conditioned matrix $\mA$ with a small condition number, the singular values are drawn from normal distribution $\nor(1, 0.04)$ and thus are centered around $1$ with high probability. For a ill-conditioned matrix $\mA,$ we set the singular values to be 
$$\displaystyle \mSig_{j,j} = 10^{(-4)\, \frac{j-1}{p-1}}, \quad \text{for~} j \in [p].$$
The condition number of $\mA$ is thus $\mSig_{1,1}/\mSig_{p,p} = 10^4.$ 
For a structured matrix $\mA,$ we instead form two factor matrices $\mF = [\vf_1, \dots, \vf_p] \in \Real^{n_1 \times p}$ and $\mG = [\vg_1, \dots, \vg_p] \in \Real^{n_2 \times p}$ respectively, following the way of generating a well-conditioned matrix and then construct $\mA$ by the rule \eqref{structured}.
To form the input vector $\vb\in\Real^{n_1n_2},$ first we generate a reference vector $\vx_{\text{ref}}\in\Real^p$ following the distribution $\nor(\mathbf{1},0.25\, \id_p)$ and a small noise vector following $\nor(\mathbf{0},10^{-2}\, \id_{n_1n_2}).$ We build the data vector $\vb$ as
\[
\vb = \mA\vx_{\text{ref}}+\text{noise}.
\]

In the numerical experiments, we perform the sketching construction $\mS$ on the aforementioned three types of linear regressions. 
We then compare the performance of $\mS$ with that of the standard Gaussian matrix. We denote $\vx^*\in\Real^{p}$ as \eqref{unconstrained_numerics}, and $\hat{\vx}\in\Real^{p}$ the sketched solution
$
\hat{\vx} = \arg\min_{\vx\in\Real^p}\|\mS\mA\vx-\mS\vb\|_2^2.
$
To measure the quality of a sketched solution $\hat{\vx}$, we define
\begin{equation}
\label{error ratio}
\text{error~ratio} = \left\vert\frac{\|\mA\hat{\vx}-\vb\|_2^2-\|\mA\vx^*-\vb\|_2^2}{\|\mA\vx^*-\vb\|_2^2}\right\vert
\end{equation}
For all numerical tests below, we calculate the error ratio as the average of $100$ independent simulations. 

In \cref{sketch_size} and \cref{rank}, we plot the error ratios respectively depending on different choices of the parameters: the sketching dimension $m$ and the number of unknown variables $p.$ We show the numerical results of different candidates for the proposed tensor-structured sketching design in \cref{GR}. The plots are all in the log-log scale.

\subsection{Dependence on the sketching size $m$}
It is of interest to test how the sketching dimension $m$ affects the quality of our sketching matrix $\mS.$ \cref{sketch_size} presents the empirical performance of $\mS$ as the sketching dimension $m$ varies. We summarize the numerical observations as follows: 
\begin{enumerate}
\item
In the small error regime (under the scale between $10^{-1}$ and $10^{-2}$), the roughly parallel lines validate the theory that the G+R sketch has the same performance as that of the standard Gaussian sketch and the sketching dimension $m$ keeps a quadratic dependence on $\eps$.

\item
In contrast, there are rapid decays in the large error ratio regime for the G+R sketching in both well-conditioned and ill-conditioned programs. Such phenomena corresponds to the quantitative shift in our sketching dimension bound shown in \cref{unconstrained}. Specifically, the error ratio decays at the order $\mathcal{O}(1/m)$ when the sketching size $m$ is small but gradually drops at a slower rate $\mathcal{O}(1/\sqrt{m})$ as $m$ grows.  

\item
Unlike its performance in the unstructured programs, the G+R sketching does not have a drastic decay in the structured program. This may suggest a better estimate of the sketching dimension for tensor-structured programs.
\end{enumerate}

\begin{figure}[!ht]
\centering
\includegraphics[width=0.55\linewidth]{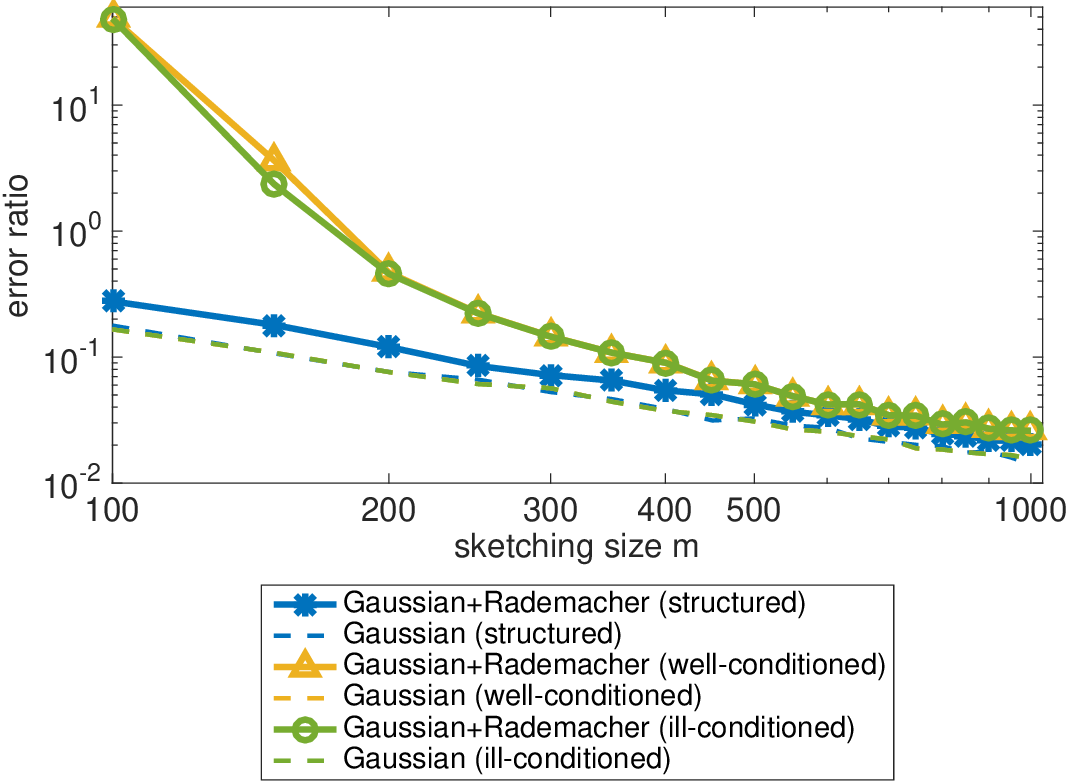}
\caption{We plot in solid curves the error ratios averaged over 100 trials of the G+R sketching with density $q=0.2$, respectively on the well-conditioned, ill-conditioned and structured linear regressions with varying sketching size $m$. In comparison, the performance of unstructured Gaussian sketching matrix is plotted in dash lines on the same three types of regression problems for reference. The size of the regression problems is set as $n_1n_2 = 64^2$ by $p=15$. }
\label{sketch_size}
\end{figure}
  
\subsection{Dependence on the number of unknowns $p$}
The matrix $\mA\in\Real^{n_1n_2\times p}$ is full rank with high probability due to our construction method, thus the number of unknowns $p=\text{rank}(\mA)$ considering $n_1,n_2\geq p$. We show in \cref{rank} the performance of the sketching strategy as the number of unknowns $p$ changes.
\begin{enumerate}
\item
When the error ratios are relatively small, i.e. under the scale $10^{-1},$ the G+R construction for all three types of programs is shown to have the same dependence as the standard Gaussian matrix on $p,$ or equivalently $\text{rank}(\mA),$ given that the slopes are rather similar. This observation is important and provides numerical support for our result \cref{unconstrained}, that the tensor-structured sketching maintains the same optimal dependence on $\text{rank}(\mA)$ as conventional Gaussian matrix in unconstrained linear regressions in the small error setting.

\item
Larger errors are observed in the solid curves for well and ill-conditioned programs in the large $p$ regime. This is consistent with our unconstrained linear regression theory \cref{unconstrained} that the error ratio $\eps$ may change the optimal dependence on $\text{rank}(\mA)$ to a quadratic dependence when the sketching dimension $m$ is fixed.
\end{enumerate}

\begin{figure}[!ht]
\centering
\includegraphics[width=0.55\linewidth]{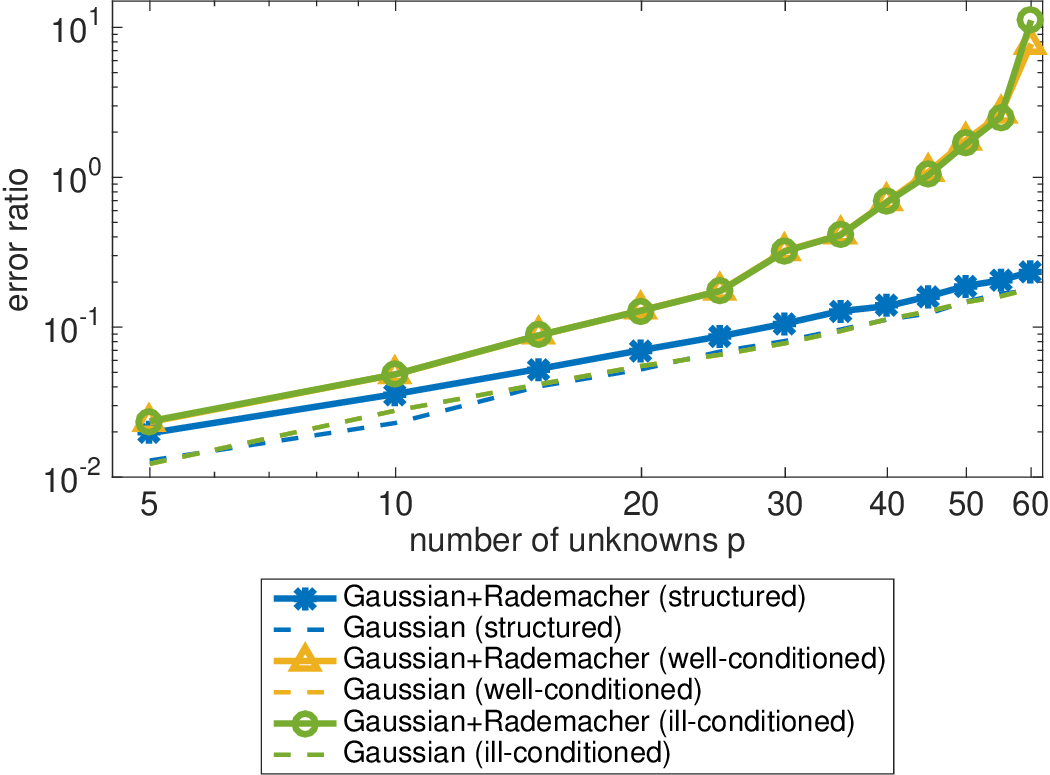}
\caption{We plot the sketching error ratios with varying number of unknown variables $p$ with fixed ambient dimension $n_1n_2=64^2$ for three different types of linear regression programs. We fix the sketching dimension $m=400$. We compare the performances of the G+R sketching constructed with density $q=0.2$ and standard Gaussian sketching.} 
\label{rank}
\end{figure}

\subsection{Performance of different sub-Gaussian distributions}
\label{BN}
In the construction of sketching matrix $\mS$ \eqref{sketch}, we assume each row is a tensor product of sub-Gaussian random vectors. So far we have employed the sketching design G+R in the numerical experiments above. In this section, we investigate other choices of sub-Gaussian distributions including Rademacher+Rademacher (R+R) and uniform $[-\sqrt{3},\sqrt{3}]$ \footnote{The matrix is constructed by independent entries following the uniform $[-\sqrt{3},\sqrt{3}]$ distribution. Such construction satisfies the condition for  $\mS$ \eqref{sketch ii} as the uniform $[-\sqrt{3},\sqrt{3}]$ distributed random variables have zero mean and unit variance.} + uniform $[-\sqrt{3},\sqrt{3}]$ (U+U) and Gaussian+Gaussian (G+G). In \cref{GR}, for the well-conditioned linear regression only, we test three sketching strategies. All three curves seem to be parallel to that of unstructured sketching matrices, though with different constant intercepts. We think this difference is caused by the factor of maximal $\psi_2$ norm $\alpha$ that is implicitly contained in the constant $C$ of \eqref{unconstrained m} in \cref{unconstrained}. In fact, the parameter $\alpha$ has different values approximately as $1.20,1.34$ and $1.63$ for Rademacher, uniform $[-\sqrt{3},\sqrt{3}]$ and Gaussian random variables respectively. Among all three choices, the R+R design turns out to have the best performance as it has the smallest factor value. 

\begin{figure}[!ht]
\centering 
\includegraphics[width=0.55\linewidth]{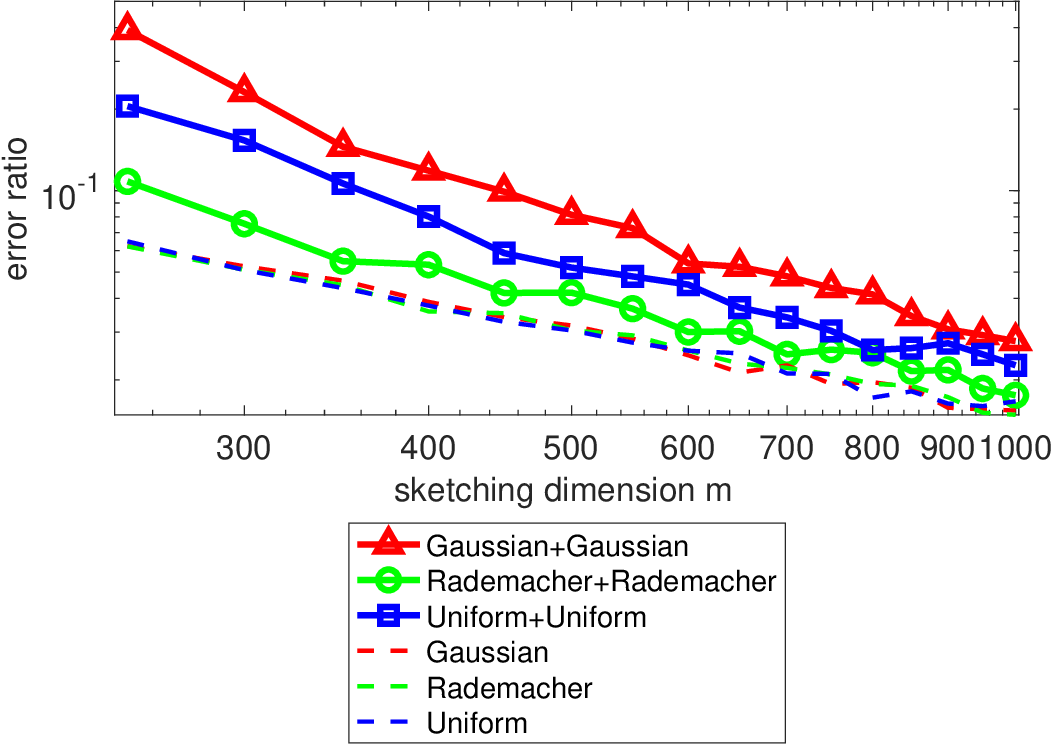}
\caption{We compare the sketching performances of the G+G, R+R and U+U sketches with varying sketching size $m$ on a well-conditioned linear regression problem of size $n_1n_2 = 64^2$ by $p=15.$ All three types of random sketches are constructed with density level $q=0.2.$ In addition, we implement the standard Gaussian, Rademacher and uniform $[-\sqrt{3},\sqrt{3}]$ matrices on the same program for reference.} 
\label{GR}
\end{figure}

\section{Conclusions and future work}
\label{sec: conclusion}
In this paper, motived by structured least squares in practical applications, we presented a row-wise tensor-structured sketching design to accelerate the solving procedure. 
For unconstrained linear regressions, we provided a sharp sketching dimension bound that is derived from the optimal JL property of the proposed sketch. We have also demonstrated a state-of-the-art sketching estimate regarding constrained least squares problems. The result involves the calculation of a $M$-complexity parameter and directly applies to common types of optimizations in different geometry landscapes.
In light of the theoretical support to the main result, we developed the analysis for the maximum embedding error of the sketching matrix over a fixed set, with the help of the generic chaining technique.
Our theories are then verified by numerical simulations on various choices of sketching matrices and parameter situations. 

Future work remains, firstly, considering the suboptimal estimate led by the newly introduced $M$-complexity compared to the classical Gaussian width, we shall seek for a more accessible complexity term to offer a better characterization for sketching constrained optimizations.
Secondly, the sketching model in this work considers only the sub-Gaussian random variables, an interesting open problem would be analyzing the sketching property for the randomized FFT-related construction with tensor structure, as such class of sketch enables fast matrix-vector multiplications.
The last open problem is inspired by the numerical observations shown in \cref{sketch_size} and \cref{rank}, that the proposed sketching design works better on the tensor-structured programs than unstructured programs. This may suggest that a tighter sketching size estimate could be derived for tensor-structured programs.

\section*{Acknowledgements}
The authors would like to thank Linjian Ma for pointing out a mistake in an old manuscript. We also thank Rachel Ward and Joe Neeman for their insights and helpful discussions. Chen is supported by the Office of Naval Research (award N00014-18-1-2354) and by the National Science Foundation (awards 1952735 and 1620472). Jin is supported by the grant NSF HDR-1934932. We also thank the anonymous referees for their careful reading of the manuscript and many insightful comments.  

\appendix
\section{Proof of \cref{lemma: M(AK+z)}}
\label{proof: M(AK+z)}
\begin{proof}[Proof of \cref{lemma: M(AK+z)}]
From \cref{M def}, let $\sT^*$ be an almost optimal skeleton of $\AKb_+ \cup  \{{\bf 0}\}$ such that
\begin{equation}
\label{optimal skeleton}
\gamma_1(\sT^*)+\gamma_2(\sT^*) \leq \frac{13}{12}\,M(\AKb_+).
\end{equation}

We then show that $\sT^* + \vz \cup \{{\bf 0}\}$ is a skeleton for $\overline{\AKb_++\vz} \cup  \{{\bf 0}\}$.
\begin{claim}
\label{claim 1}
\begin{equation}
\label{skeleton}
\overline{\AKb_++\vz} \cup \{{\bf 0}\} \subset \conv(\sT^* + \vz \cup \{{\bf 0}\} ).
\end{equation}
\end{claim}

\begin{proof}[Proof of \cref{claim 1}]
For any $\vy \in \overline{\AKb_++\vz},$ there exists an element $\vy'$ in $\AKb_+$ such that $\vy = (\vy'+\vz)/\|\vy'+\vz\|_2.$ Since $\vy' \in \AKb_+ \subset \conv(\sT^*),$ suppose $\vy'$ has the expression:
$
\vy' = \sum_{i=1}^I \alpha_i {\bf{t}}_i,
$
where $\{{\bf{t}}_i\}_{i \in [I]} \subset \sT^*, \alpha_i \geq 0, \sum_{i=1}^I \alpha_i = 1$. Then we have a convex construction for $\vy$ from the set $\sT^*+\vz \cup \{\bf 0\}:$
\[
\vy = \frac{\vy'+\vz}{\|\vy'+\vz\|_2} = \sum_{i=1}^I \frac{\alpha_i}{\|\vy'+\vz\|_2} ({\bf{t}}_i+\vz) + (1- \frac{1}{\|\vy'+\vz\|_2}) \, \bf 0,
\]
where the components $\{{\bf{t}}_i + \vz\}_{i \in [I]}, {\bf 0}$ belong to $\sT^*+\vz \cup \{\bf 0\}$ and positive weights $(\sum_{i=1}^I \frac{\alpha_i}{\|\vy'+\vz\|_2})+(1- \frac{1}{\|\vy'+\vz\|_2}) = 1$. Note that $\|\vy'+\vz\|_2 = \sqrt{\|\vy'\|_2^2+\|\vz\|_2^2 + 2\, \langle \vy',\vz\rangle} \geq \sqrt{2}>1,$ so $1/\|\vy'+\vz\|_2 < 1.$

The proof of \cref{claim 1} is complete.
\end{proof} 
Based on \eqref{M-complexity}, we then compute $M(\overline{\AKb_++\vz})$ by the $\gamma$-functionals of its skeleton: $\sT^*+\vz \cup \{\bf 0\}.$ 

Next, we aim to bound $\gamma(\sT^*+\vz \cup \{\bf 0\})$ by $\gamma(\sT^*)$ up to a constant. See the following claim.
\begin{claim}
\label{claim 2}
\[
\gamma_\alpha(\sT^*+\vz \cup \{{\bf 0}\}) \leq 12\, \gamma_\alpha(\sT^*), ~~\text{for}~~\alpha = 1, 2.
\]
\end{claim}

\begin{proof}[Proof of \cref{claim 2}]
By the sub-additivity of $\gamma$-functionals as proved in Lemma 2.1 from \cite{VL19}, we know
\begin{equation}
\label{additivity} 
\begin{array}{ll}
\gamma_\alpha(\sT^*+\vz \cup \{{\bf 0}\})& \leq 3\, (\gamma_\alpha(\sT^*+\vz) + \gamma_\alpha(\{{\bf 0}\}) + \diam(\sT^*+\vz \cup \{{\bf 0}\}))\\
& \leq 3\, (\gamma_\alpha(\sT^*) + \diam(\sT^*+\vz \cup\{{\bf 0}\})),
\end{array}
\end{equation}
due to the fact that $\gamma$-functionals are translation-invariant and $\gamma_\alpha(\{{\bf 0}\}) = 0.$

To evaluate $\diam(\sT^*+\vz \cup \{{\bf 0}\})),$ we first have
\[
\diam(\sT^*+\vz \cup \{{\bf 0}\})) \leq \diam(\sT^*+\vz)+\text{dist}({\bf 0}, \sT^*+\vz ) .
\]
As diameter is translation-invariant, 
\[
 \diam(\sT^*+\vz) =\diam(\sT^*).
 \]
Also, 
\begin{equation*}
\begin{array}{ll}
\text{dist}({\bf 0}, \sT^*+\vz ) &= \text{dist}({\bf -z}, \sT^* ) \\
& \leq  \text{dist}(-\vz,{\bf 0})+\text{dist}({\bf 0} ,\sT^*)   
 =  \text{dist}({\bf 0}, \conv(\sT^*))+\|\vz\|_2  \\
&\leq \text{dist}({\bf 0}, \AKb_+)+\|\vz\|_2 = 2, \quad \text{since}~\AKb_+\subset  \conv(\sT^*).
\end{array}
\end{equation*}
Hence,
\begin{equation}
\label{diam T+z 0}
\diam(\sT^*+\vz \cup \{{\bf 0}\})) \leq \diam(\sT^*) + 2 \leq 3\, \diam(\sT^*).
\end{equation}
$\diam(\sT^*)\geq 1$ is because $\diam(\sT^*)= \diam(\conv(\sT^*)) \geq \diam(\AKb_+ \cup \{{\bf 0}\}) =1.$ 

Finally, we combine the results of \eqref{additivity} and \eqref{diam T+z 0} and obtain
\begin{equation*}
\label{gamma_alpha}
\begin{array}{ll}
\gamma_\alpha(\sT^*+\vz \cup \{{\bf 0}\}) &\leq 3\, (\gamma_\alpha(\sT^*) + \diam(\sT^*+\vz \cup\{{\bf 0}\}))\\
& \leq 3\, (\gamma_\alpha(\sT^*) + 3\,\diam(\sT^*)) \leq 12\, \gamma_\alpha(\sT^*).
\end{array}
\end{equation*}
Note that a set's diameter is always less than its $\gamma$-functionals.

The proof of \cref{claim 2} is complete.
\end{proof}

Therefore, by \cref{claim 1}, \cref{claim 2} and \eqref{optimal skeleton}, we have
\[
M(\overline{\AKb_++\vz}) \leq (\gamma_1+\gamma_2)(\sT^*+\vz \cup \{{\bf 0}\}) \leq 12\, (\gamma_1+\gamma_2)(\sT^*) = 13\, M(\AKb_+). 
\]
The entire proof for \cref{lemma: M(AK+z)} is complete.
\end{proof}

\bibliographystyle{siamplain}
\bibliography{references}

\end{document}